\documentclass{article}
\usepackage[utf8]{inputenc}
\usepackage[margin=1 in]{geometry}
\usepackage{amsfonts, amssymb}
\usepackage{mathtools}
\usepackage{xcolor}
\usepackage{stmaryrd}
\definecolor{webgreen}{rgb}{0,.5,0}
\definecolor{darkblue}{rgb}{0.0,0,0.7}
\usepackage[colorlinks=true,
linkcolor=black, citecolor=black, urlcolor=black]{hyperref}
\usepackage{amsmath}
\usepackage{amsthm}
\usepackage{enumitem}
\providecommand\ldb{\llbracket}
\providecommand\rdb{\rrbracket}

\newcommand{\newword}[1]{\textcolor{blue}{\textbf{\textit{#1}}}}

\newtheorem{theorem}{Theorem}
\numberwithin{theorem}{section}
\newtheorem{lemma}[theorem]{Lemma}
\newtheorem{prop}[theorem]{Proposition}
\newtheorem{cor}[theorem]{Corollary}

\theoremstyle{definition}
\newtheorem{definition}[theorem]{Definition}
\newtheorem{remark}[theorem]{Remark}
\newtheorem{example}[theorem]{Example}
\newtheorem{algorithm}[theorem]{Algorithm}

\theoremstyle{remark}

\newcommand{\mb}{\mathbb}

\newcommand{\mc}{\mathcal}
\newcommand{\tn}{\textnormal}
\newcommand{\se}{\subseteq}
\newcommand{\ol}{\overline}
\newcommand{\lam}{\lambda}

\newcommand{\Lam}{\Lambda}
\newcommand{\mf}{\mathfrak}

\newcommand{\bs}{\backslash}

\newcommand{\csf}[1]{X_{#1}}
\newcommand{\ksf}[1]{\overline{X}_{#1}}
\newcommand{\augm}[1]{\widetilde{m}_{#1}}
\newcommand{\kaugm}[1]{\overline{\widetilde{m}}_{#1}}
\newcommand{\m}{\ol{\widetilde{m}}}
\newcommand{\odotprod}{\odot}

\newcommand{\WGraphs}{\textsf{\textup{WGraphs}}}
\newcommand{\mWGraphs}{\mf{m}\textsf{\textup{WGraphs}}}
\newcommand{\Sym}{\textsf{\textup{Sym}}}
\newcommand{\mSym}{\mf{m}\textsf{\textup{Sym}}}

\usepackage[maxbibnames=10,style=alphabetic,maxalphanames=10]{biblatex}
\title{New interpretations for Kromatic symmetric function expansions}

\author{Aarush Kulkarni \footnote{Acton-Boxborough Regional High School. Email: \texttt{\href{mailto:aarushkulkarni26@gmail.com}{aarushkulkarni26@gmail.com}}} \and Laura Pierson \footnote{Department of Mathematics, Harvard University. Email: \texttt{\href{mailto:lcpierson73@gmail.com}{lcpierson73@gmail.com}}}}

\begin{document}
\maketitle

\begin{abstract}
The \newword{Kromatic symmetric function (KSF)} $\ksf{G}$, introduced by Crew, Pechenik, and Spirkl \cite{crew2023kromatic}, is a $K$-theoretic analogue of the \newword{chromatic symmetric function (CSF)} $\csf{G}$. We study expansion formulas for the KSF in two different bases. First, we explore recursive ways to compute the KSF's expansion in the \newword{$K$-theoretic monomial symmetric function} basis $\ol{\widetilde{m}}_\lam$ from \cite{crew2023kromatic}, generalizing formulas that were used in \cite{pierson2025graphs} to show that the KSF distinguishes certain families of graphs that are not distinguished by the CSF. Second, we give new interpretations for the KSF's expansion in the \newword{$K$-theoretic power sum basis} $\ol{p}_\lam$, using an inclusion-exclusion approach like the one in \cite{stanley1995symmetric} instead of the acyclic orientation approach from \cite{pierson2025lyndon}. Finally, we give $K$-analogues of Schmitt's and of Humpert and Martin's antipode formulas for the Hopf algebra of graphs, along with a Hopf algebra interpretation for both the $\ol{p}$-expansion and $\m$-expansion of the KSF.

\end{abstract}

\section{Introduction}
\label{sec:introduction}

The \newword{Kromatic symmetric function (KSF)} was introduced by Crew, Pechenik, and Spirkl in \cite{crew2023kromatic} as a $K$-theoretic analogue for Stanley's \newword{chromatic symmetric function (CSF)}. While the CSF $X_G$ encodes \newword{proper colorings} of the graph $G$ that assign each vertex a single color such that adjacent vertices get distinct colors, the KSF instead encodes \newword{proper set colorings} that assign each vertex a nonempty set of colors such that adjacent vertices get disjoint color sets. While the KSF does not actually have a known $K$-theoretic interpretation, it is still of algebraic interest because it has a nice interpretation in terms of Hopf algebras \cite{marberg2023kromatic}, and it is of combinatorial interest because it has elegant expansion formulas using $K$-analogues of the Schur functions $s_\lam$ \cite{crew2023kromatic}, the monomial symmetric functions $m_\lam$ \cite{crew2023kromatic}, the elementary symmetric functions $e_\lam$ \cite{marberg2023kromatic}, and the power sum symmetric functions  $p_\lam$ \cite{pierson2025lyndon, pierson2024power}. It is also interesting as a graph invariant because it contains strictly more information about a graph than its CSF, while still failing to distinguish all graphs \cite{crew2023kromatic, pierson2025graphs}.

After introducing some relevant definitions in \S\ref{sec:background}, we explore two new interpretations for the $\ol{p}$-expansion of the KSF:
\begin{itemize}
     \item \S\ref{sec:monomial}: In \cite{pierson2025graphs}, the second author and Samanta developed recursive formulas for the $\m_\lam$-expansion of the KSF for certain families of graphs, as tools to show that the KSF distinguishes those graphs from each other even though the CSF does not. We found those formulas algebraically interesting in their own right because they make use of Tsujie's $\odot$ multiplication from \cite{tsujie2018chromatic} in new ways, so we dive further into that algebra and generalize those formulas in two different ways. Our new formulas could similarly be used to build more families of graph pairs with equal CSF but different KSF.

    \item \S\ref{sec:power-sum}: Crew, Pechenik, and Spirkl asked in \cite{crew2023kromatic} whether the power sum expansion for the KSF could be computed using an inclusion-exclusion approach similar to the one used for the CSF in \cite{stanley1995symmetric}. We address this question by giving an inclusion-exclusion based recursion for the $\ol{p}$-expansion for the KSF.

    \item \S\ref{sec:hopf}: We give $K$-analogues of the antipode formulas of Schmitt \cite{schmitt1994incidence} and of Humpert and Martin \cite{humpert2012incidence}, and Hopf algebra interpretations for the $\ol{p}$-expansion formula from \cite{pierson2025lyndon}.
\end{itemize}


\section{Background}
\label{sec:background}


As is customary, $\mathbb{N}$ will denote the set of positive integers. We write $[n]$ for the set of positive integers $\{1,2,\dots,n\}.$

A \newword{weighted graph} $G$ is a finite set of vertices $V(G)$ and a set of edges $E(G)$ that are distinct unordered pairs of distinct vertices, together with a weight function $w:V(G)\to\mb{Z}_{>0}$ assigning a positive integer weight to all vertices. Vertices with an edge between them are \newword{adjacent}. Weighted graphs are often written as $(G,w)$, but we will assume that $w$ is incorporated into the data of $G$. A graph is \newword{unweighted} if all weights are equal to 1. Throughout this paper, graphs are simple, and unless otherwise stated, unweighted.

An \newword{independent set} or \newword{stable set} is a subset of the vertices in which no two vertices are adjacent. For a subset $S\se V(G)$, we write $G|_S$ for the \newword{induced subgraph} of $G$ with vertex set $S$, i.e.\ the subgraph with vertex set $S$ and edge set consisting of all edges in $E(G)$ with both endpoints in $S$. For a subset $F\se E(G)$, we write $G_F$ for the subgraph of $G$ with the full vertex set $V(G)$ but with edge set $F$. We say that $F\se E(G)$ is a \newword{flat} if adding any edge of $E(G)\setminus F$ to $G_F$ would decrease the number of connected components. An \newword{acyclic orientation} on $G$ is an assignment of a direction to all edges of $G$ such that no directed cycles are formed. We write $G/F$ for the graph formed by contracting all edges in $F$, where \newword{contracting} an edge $uv$ means deleting the edge and replacing its two endpoints $u$ and $v$ with a single merged vertex $uv$ of weight $w(uv)=w(u)+w(v)$, connected to all vertices in $G$ that were adjacent to at least one of $u$ or $v$.

The \newword{empty graph} $\varnothing$ is the graph with no edges and no vertices. The \newword{disjoint union} $G\sqcup H$ of two graphs $G$ and $H$ is the graph with vertex set $V(G)\sqcup V(H)$ and edge set $E(G)\sqcup E(H)$. The \newword{join} $G\odotprod H$ is the graph that additionally has all edges connecting a vertex of $G$ to a vertex of $H$. Two graphs $G$ and $H$ are \newword{isomorphic} if there is a weight-preserving bijection from $V(G)$ to $V(H)$ that also induces a bijection from $E(G)$ to $E(H)$, so two vertices are adjacent in $G$ if and only if their images are adjacent in $H$.

    The algebraic setting for our invariants is the \newword{ring of symmetric functions}, denoted $\Lambda$. This is a subring of $\mathbb{Q} \ldb x_1,x_2,\dots \rdb $ consisting of the power series of bounded degree that are invariant under any permutation of the variables, and we write $\widehat{\Lam}$ for the ring of all such power series, without the bounded degree condition. A \newword{partition} $\lambda=(\lambda_1,\dots,\lambda_\ell)$ is a sequence of non-increasing positive integers. The integers $\lambda_i$ are the \newword{parts} of $\lambda$, $\ell(\lambda)$ is the \newword{length}, and $|\lambda|=\sum\lambda_i$ is the \newword{size}. If $|\lambda|=n$, we write $\lambda\vdash n$. We primarily use the \newword{augmented monomial symmetric functions}, $\{\augm{\lambda}\}$, as a basis for $\Lambda$, where $\augm{\lambda}:=(\prod_{i\ge 1}r_i(\lambda)!)m_\lambda$ and $r_i(\lambda)$ is the number of parts of $\lambda$ of size $i$ and $m_{\lambda}$ is the monomial symmetric function.

A \newword{proper coloring} of a graph $G$ is a function $\kappa\colon V\to\mathbb{N}$ such that $\kappa(u)\neq\kappa(v)$ for all $\{u,v\}\in E$. The \newword{chromatic symmetric function (CSF)} \cite{stanley1995symmetric} of $G$ is $$ \csf{G}(x_1,x_2,\dots)=\sum_{\kappa\text{ proper coloring}}\ \prod_{v\in V}x_{\kappa(v)}^{w(v)}. $$

\begin{theorem}[\cite{stanley1995symmetric}]
    The coefficient of $\augm{\lambda}$ is the number of ways to partition the vertices of $G$ into stable (independent) sets whose sizes are the parts of $\lambda$. This gives the expansion $$ \csf{G}=\sum_{\lambda\vdash|V(G)|}|\textsf{\textup{Stab}}_\lambda(G)|\cdot \augm{\lambda} $$ where $|\textsf{\textup{Stab}}_\lambda(G)|$ is the number of partitions of $V(G)$ into stable sets, each called a \newword{stable set partition}, with sizes matching the parts of $\lambda$.
\end{theorem}
The CSF interacts nicely with graph operations. The disjoint union $G\sqcup H$ and join $G\odotprod H$ satisfy $\csf{G\sqcup H}=\csf{G}\cdot\csf{H}$ and $\csf{G\odotprod H}=\csf{G}\odotprod\csf{H}$, where the \newword{join product} $\odotprod$, introduced by Tsujie in \cite{tsujie2018chromatic}, acts on the augmented monomial basis by $\augm{\lambda}\odotprod\augm{\mu}=\augm{\lambda\sqcup\mu}$, where $\lambda\sqcup\mu$ is the partition obtained by taking all parts of $\lambda$ together with all parts of $\mu$ (so for instance, $321\sqcup 211=322111$). Tsujie's main motivation for introducing the $\odotprod$ multiplication was to make it easier to describe ways to combine stable set partitions in order to prove that the CSF distinguishes a class of graphs known as \newword{trivially perfect graphs}.

A \newword{proper set coloring} of $G$ is a function $\kappa\colon V\to 2^{\mathbb{N}}\setminus\{\varnothing\}$ with each $\kappa(v)$ finite such that $\kappa(u)\cap\kappa(v)=\varnothing$ for all $\{u,v\}\in E$. The \newword{Kromatic symmetric function (KSF)} \cite{crew2023kromatic} of $G$ is $$ \ksf{G}:=\sum_{\kappa\text{ proper set coloring}}\ \prod_{v\in V(G)}\prod_{i\in\kappa(v)}x_i^{w(v)}\in\widehat{\Lam}. $$ The KSF was introduced as a $K$-theoretic analogue of the CSF. It contains more information about $G$ than the CSF does; its lowest degree terms match $\csf{G}$, but it also has additional higher degree terms. This allows it to distinguish some graphs that the CSF cannot \cite{crew2023kromatic}. However, it is not a complete graph invariant, as pairs of nonisomorphic graphs with the same KSF have been found \cite{pierson2025graphs}. The KSF has an expansion formula using a $K$-analogue of the $\widetilde{m}_\lam$ basis, called the \newword{$K$-theoretic augmented monomial symmetric functions} $\kaugm{\lambda}$, which can be defined by $\kaugm{\lambda}:=\ksf{K_\lambda}$, where $K_\lambda$ denotes the complete graph on $\ell(\lambda)$ vertices whose $i$th vertex has weight $\lambda_i$. The expansion is in terms of \newword{stable set covers}, collections of distinct nonempty stable sets $\mathcal{C}=\{I_1,\dots,I_k\}$ whose union is $V(G)$:
\begin{theorem}[\cite{crew2023kromatic}]
    The expansion of $\ksf{G}$ in the $K$-theoretic basis $\{\kaugm{\lambda}\}$ is $$ \ksf{G}=\sum_{\mathcal{C}\in\textsf{\textup{SSC}}(G)}\kaugm{\lambda(\mathcal{C})}, $$ where $\textsf{\textup{SSC}}(G)$ is the set of all stable set covers of $G$, and $\lambda(\mathcal{C})$ is the partition given by the total weights of the stable sets in $\mathcal{C}$.
\end{theorem}

The join product $\odotprod$ has a $K$-theoretic counterpart, which is the version used throughout Section~\ref{sec:monomial}. It acts on the $K$-theoretic augmented monomial basis by $\kaugm{\lambda}\odotprod\kaugm{\mu}:=\kaugm{\lambda\sqcup\mu}$, extended bilinearly, and it satisfies $\ksf{G\odotprod H}=\ksf{G}\odotprod\ksf{H}$ for the graph join $G\odotprod H$, since no stable set of $G\odotprod H$ can meet both sides. Its identity element is $\kaugm{\varnothing}=1$, the basis element indexed by the empty partition.

Many of the formulas below involve $\odotprod$-inverses and negative $\odotprod$-powers, so we note down once and for all the setting in which these make sense. Since the lowest degree part of $\kaugm{\lambda}$ is $\augm{\lambda}$, every element of $\widehat{\Lam}$ is a unique, possibly infinite, linear combination $\sum_\lambda c_\lambda\kaugm{\lambda}$. Then $(\widehat{\Lam},\odotprod)$ is a commutative ring with identity $1$, graded by $|\lambda|$, and an element of $\widehat{\Lam}$ is invertible with respect to $\odotprod$ if and only if its constant term is nonzero: if $g$ has no constant term, then
$$ (1+g)^{\odotprod(-1)}=\sum_{k\ge 0}(-1)^kg^{\odotprod k}, $$
which is a well defined element of $\widehat{\Lam}$ because $g^{\odotprod k}$ has no terms of degree less than $k$. For $n\ge 0$ we write $f^{\odotprod n}$ for the $n$-fold $\odotprod$ product of $f$ with itself, with $f^{\odotprod 0}:=1$, and for $n<0$ we write $f^{\odotprod n}:=(f^{\odotprod(-1)})^{\odotprod(-n)}$, which is defined whenever $f$ has nonzero constant term. More generally, for $a\in\mb{Z}$ and $g$ with no constant term we write
$$ (1+g)^{\odotprod a}:=\sum_{k\ge 0}\binom{a}{k}g^{\odotprod k}, $$
in agreement with the above, and these powers obey the usual exponent rules. Every quotient appearing in Section~\ref{sec:monomial} is an $\odotprod$-quotient of this kind, and in each case the denominator has constant term 1.

It is also convenient to name the generating function that counts stable set covers without insisting on which vertices get covered.

\begin{definition}\label{def:Y}
For a graph $K$, let
$$ Y_K:=\sum_{S\se V(K)}\ksf{K|_S} $$
be the generating function for stable set covers of all induced subgraphs of $K$. For an induced subgraph $K=H|_C$ we abbreviate this as $Y_{H|_C}=\sum_{S\se C}\ksf{H|_S}$. The term coming from $S=\varnothing$ shows that $Y_K$ has constant term 1, so $Y_K$ is invertible with respect to $\odotprod$.
\end{definition}

\section{Recursive monomial expansion formulas}
\label{sec:monomial}

Motivated by the formulas used in \cite{pierson2025graphs} to build families of graphs distinguished by the KSF but not the CSF, in this section we explore what the $\odot$ multiplication lets us do, by deriving recursive formulas for the chromatic and Kromatic symmetric functions of graphs built by gluing smaller subgraphs along a shared intersection.

The setup throughout this section is the following. We take graphs $G_1,\dots,G_d$ that share a common induced subgraph $H$, in the sense that $G_i|_{V(H)}=H$ for all $i$ and $V(G_i)\cap V(G_j)=V(H)$ for all $i\ne j$, and we write $V_i':=V(G_i)\setminus V(H)$ for the vertices of $G_i$ outside the shared subgraph, and $V':=\bigcup_{i=1}^d V_i'$ for all of the vertices outside $H$. The glued graph always has vertex set $\bigcup_{i=1}^d V(G_i)$ and always keeps all of the edges $\bigcup_{i=1}^d E(G_i)$; the two gluings differ in what happens between $V_i'$ and $V_j'$ for $i\ne j$. In \S\ref{subsec:join-type} we add every edge $\{u,v\}$ with $u\in V_i'$ and $v\in V_j'$ for $i\ne j$, so that the vertices outside $H$ coming from different $G_i$ are completely adjacent to one another, and in \S\ref{subsec:disjoint-type} we add no edges between them at all. Each construction is spelled out again in the first theorem where it is used.

The connection between the two invariants throughout this section is the following basic relation, which is well known, so we do not include a proof.
\begin{remark}
\label{rem:ksf_to_csf}
The lowest degree homogeneous component of the Kromatic symmetric function $\ksf{G}$ is the chromatic symmetric function $\csf{G}$.
\end{remark}

Using the above remark, we can simply isolate the terms of degree $|V(G)|$ in an expression for $\ksf{G}$, and what is left over is the corresponding expression for $\csf{G}$.

\subsection{Join-type constructions}\label{subsec:join-type}

We begin with the first of the two constructions, which generalizes the graph join $A\odot B$, where every vertex of $A$ is adjacent to every vertex of $B$. The idea is to understand the chromatic symmetric function of a graph obtained by joining several subgraphs $G_i$ along a common intersection $H$, with all the vertices outside $H$ coming from different $G_i$ made pairwise adjacent. To bring out the logic of the inclusion-exclusion arguments, we first work through the case where the intersection is a single vertex $v$, which was studied in \cite{pierson2025graphs}:

\begin{example} \label{prop:single_vertex_join}
    Let $G_1$ and $G_2$ be graphs such that $V(G_1)\cap V(G_2)=\{v\}$ and let $G$ be the graph formed by the union of $G_1$ and $G_2$ along with all edges connecting $V(G_1)\setminus\{v\}$ to $V(G_2)\setminus\{v\}$. Then $$ \csf{G}=\csf{G_1}\odot\csf{G_2\setminus v}+\csf{G_1\setminus v}\odot\csf{G_2}-\augm{1}\odot\csf{G_1\setminus v}\odot\csf{G_2\setminus v}. $$

    To see this, consider the stable set partitions of $G$. Since every vertex of $V(G_1)\setminus\{v\}$ is adjacent to every vertex of $V(G_2)\setminus\{v\}$, no stable set can reach into both sides at once, so the entire partition is determined by how the shared vertex $v$ is covered. If the stable set containing $v$ uses only vertices of $G_1$ besides $v$ itself, then we are looking at a partition of $G_1$ together with a partition of $G_2\setminus v$, and these are counted by $\csf{G_1}\odot\csf{G_2\setminus v}$. Symmetrically, if that stable set stays inside $G_2$, we get $\csf{G_1\setminus v}\odot\csf{G_2}$.

    Adding these two cases double counts exactly the partitions in which $\{v\}$ is its own singleton stable set, since such a partition looks valid from either side. By inclusion-exclusion we subtract one copy of this overlap: when $\{v\}$ is a singleton, its generating function is $\augm{1}$, and the remaining vertices of $G_1\setminus v$ and $G_2\setminus v$ are partitioned independently, so the term to remove is $\augm{1}\odot\csf{G_1\setminus v}\odot\csf{G_2\setminus v}$. This gives the desired formula.
\end{example}

We now treat the general case, where the intersection $H$ is arbitrary.

\begin{theorem} \label{thm:csf-general}
    Let $G_1,\dots,G_d$ be graphs sharing a common induced subgraph $H$, so that $G_i|_{V(H)}=H$ for all $i$ and $V(G_i)\cap V(G_j)=V(H)$ for all $i\ne j$, and write $V_i':=V(G_i)\setminus V(H)$. Let $G$ be the graph on the vertex set $\bigcup_{i=1}^d V(G_i)$ whose edges are the edges of the $G_i$ together with all pairs $\{u,v\}$ such that $u\in V_i'$ and $v\in V_j'$ for some $i\ne j$. Then, the CSF of $G$ is given by $$ \csf{G}=\sum_{C_0\sqcup\dots\sqcup C_d=V(H)}\left[\left(\sum_{\lambda}(-d+1)^{\ell(\lambda)}\augm{\lambda}\right)\odot\left(\bigodot_{i=1}^d\csf{G_i|_{V_i'\cup C_i}}\right)\right], $$ where the inner sum is taken over all stable set partitions of $H|_{C_0}$, with $\lambda\vdash|C_0|$ being the integer partition corresponding to each such set partition.
\end{theorem}

\begin{proof}
Since each vertex of $V_i'$ is adjacent to every vertex of $V_j'$ for $j\ne i$, no stable set can use vertices from two different $V_i'$ at once, so every stable set lies inside $V(H)\cup V_i'$ for a single $i$, or else inside $V(H)$ alone. A stable set partition of $G$ is therefore determined by a partition $C_0\sqcup C_1\sqcup\dots\sqcup C_d$ of $V(H)$, noting which $V_i'$ (if any) each shared vertex's stable set reaches, together with a stable set partition of each $G_i|_{V_i'\cup C_i}$; this is what the formula sums over. The coefficient $(-d+1)^{\ell(\lambda)}$ is then exactly the inclusion-exclusion correction for the fact that a stable set living entirely inside $H$ could just as well have been assigned to any one of the $d$ sets $V_1',\dots,V_d'$.

To prove the formula, it is enough to check that every stable set partition $\mathcal{P}$ of $G$ is counted with total coefficient exactly $1$ on the right-hand side. Each stable set of $\mathcal{P}$ meets at most one of the sets $V_i'$, so $\mathcal{P}$ induces a partition $V(H)=V_0\sqcup V_1\sqcup\dots\sqcup V_d$ of the shared vertices, where for $i\in[d]$ we put $v\in V_i$ if the stable set of $\mathcal{P}$ containing $v$ also uses at least one vertex of $V_i'$, and we put $v\in V_0$ if that stable set lies entirely inside $V(H)$. Write $\mathcal{P}_0$ for the collection of stable sets of $\mathcal{P}$ that lie inside $V_0$.

Now we ask which terms on the right-hand side, indexed by a partition $C_0\sqcup\dots\sqcup C_d=V(H)$, can possibly contribute the monomial of $\mathcal{P}$. Such a term contributes precisely when each vertex lands in its own part, so that $V_i\se C_i$ for $i\in[d]$, and each stable set $S\in\mathcal{P}_0$ lies entirely in a single part $C_k$. Because the placement of each piece of $\mathcal{P}$ can be made independently, the total coefficient factors as a product over the pieces.

For a vertex $v\in V_i$ with $i\in[d]$ there is no choice at all: $v$ must go in $C_i$, contributing a factor of $1$. All the bookkeeping comes from the stable sets in $\mathcal{P}_0$, say $\mathcal{P}_0=\{S_1,\dots,S_m\}$. A contributing term counts, for each $S_k$, which part $C_j$ it was placed in; we encode this by a function $f\colon\mathcal{P}_0\to\{0,1,\dots,d\}$. The sets sent to $C_0$ make up the stable set partition of $H|_{C_0}$, which has $|f^{-1}(0)|$ parts, so that term carries the coefficient $(-d+1)^{|f^{-1}(0)|}$. Summing over all $f$, and organizing by the subset $\mathcal{A}\se\mathcal{P}_0$ of sets that $f$ sends to $C_0$, the remaining $m-|\mathcal{A}|$ sets are each free to go to any of $C_1,\dots,C_d$. Hence the total contribution from $\mathcal{P}_0$ is $$ \sum_{\mathcal{A}\se\mathcal{P}_0}d^{\,m-|\mathcal{A}|}(-d+1)^{|\mathcal{A}|} =\sum_{k=0}^m\binom{m}{k}d^{\,m-k}(-d+1)^k =\bigl(d+(-d+1)\bigr)^m=1^m=1. $$ The binomial theorem is doing exactly the inclusion-exclusion we wanted because the $+d$ counts the ways to assign a set lying inside $H$ to one of the $d$ graphs $G_i$, and the $-(d-1)$ cancels the overcounting from also allowing it in $C_0$. Since every piece of $\mathcal{P}$ contributes a factor of $1$, the monomial of $\mathcal{P}$ appears with coefficient $1$, as needed.
\end{proof}

We now extend these ideas to the Kromatic symmetric function, beginning with a lemma that will feed directly into the main theorem.

\begin{lemma}
    \label{lem:tc_identity}
    Let $G$, $H$ and $V_1',\dots,V_d'$ be as in Theorem~\ref{thm:csf-general}, let $V':=\bigcup_{i=1}^d V_i'$, and let $C\se V(H)$. With $Y_{H|_C}$ as in Definition~\ref{def:Y}, the expression $$ \left(\bigodot_{i=1}^d\frac{\sum_{S\se C}\ksf{G_i|_{V_i'\cup S}}}{Y_{H|_C}}\right)\odot Y_{H|_C} $$ is the generating function for all stable set covers of all induced subgraphs of $G$ on vertex sets of the form $V'\cup S$, where $S\se C$. That is, it equals $\sum_{S\se C}\ksf{G|_{V'\cup S}}$.
\end{lemma}

\begin{proof}
    The idea is to read each factor in the expression as a generating function for a piece of a stable set cover, and to check that multiplying the pieces with the join product $\odot$ recombines them into stable set covers of $G$ on vertex sets of the form $V'\cup S$.

    We start with the denominator: by definition $Y_{H|_C}=\sum_{S\se C}\ksf{H|_S}$ is the generating function for stable set covers of every induced subgraph of $H$ whose vertex set lies inside $C$. In the same way, the numerator $\sum_{S\se C}\ksf{G_i|_{V_i'\cup S}}$ is the generating function for stable set covers of the induced subgraphs of $G_i$ on vertex sets $V_i'\cup S$ with $S\se C$.

    Now consider the ratio $\frac{\sum_{S\se C}\ksf{G_i|_{V_i'\cup S}}}{Y_{H|_C}}$. To see what it counts, note that any stable set cover of some $G_i|_{V_i'\cup S}$ splits, in exactly one way, into the stable sets that meet $V_i'$ and the stable sets that lie entirely inside $C$. The first group is a collection $\mathcal{K}$ of stable sets of $G_i$ with $\bigcup_{I\in\mathcal{K}}I\supseteq V_i'$, with $\bigcup_{I\in\mathcal{K}}I\se V_i'\cup C$, and with no $I\in\mathcal{K}$ contained in $C$; the second group is a stable set cover of an induced subgraph of $H$ on a subset of $S$. Summing over the second group rebuilds the factor $Y_{H|_C}$, so dividing it out leaves exactly the generating function for the first group. Hence the ratio enumerates these collections, each of which covers $V_i'$ and uses no stable set contained in $C$.

    Finally, we multiply the $d$ ratios and the one factor of $Y_{H|_C}$ together with $\odot$. A term in the expansion picks one collection from each ratio and one cover from $Y_{H|_C}$ and takes their union. Because of the added edges between the different $V_i'$, every stable set of $G$ lies inside a single $V(G_i)$, so a union of stable sets drawn from the separate $V_i'$ (together with sets inside $H$) is automatically a collection of stable sets of $G$. The collections drawn from the ratios cover all of $V'=\bigcup_i V_i'$, and the $Y_{H|_C}$ factor supplies a cover of some $S\se C$. So the whole expression is the generating function for stable set covers of $G|_{V'\cup S}$ over all $S\se C$, namely $\sum_{S\se C}\ksf{G|_{V'\cup S}}$, as claimed.
\end{proof}

With the lemma in hand, each subset $C\se V(H)$ gives us the generating function for the stable set covers of $G$ that reach the shared vertices only inside $C$, with all of $V'$ covered in every case. What we actually want are the covers reaching all of $V(H)$, and those can be isolated from the others by an alternating sum over $C$:

\begin{theorem} \label{thm:ksf-general}
    Let $G$, $H$ and $V_1',\dots,V_d'$ be as in Theorem~\ref{thm:csf-general}. With $Y_{H|_C}$ as in Definition~\ref{def:Y}, the KSF of $G$ is given by $$ \ksf{G}=\sum_{C\se V(H)}(-1)^{|V(H)\setminus C|}\left[\left(\bigodot_{i=1}^d\frac{\sum_{S\se C}\ksf{G_i|_{V'_i\cup S}}}{Y_{H|_C}}\right)\odot Y_{H|_C}\right]. $$
\end{theorem}

\begin{proof}
    Once Lemma~\ref{lem:tc_identity} is in hand, the proof is a clean inclusion-exclusion: the term indexed by $C$ collects every cover of $V(H)$ that are only within $C$, and the alternating signs remain only for the covers using all of $V(H)$.

    Fix any stable set cover $\mathcal{C}^*$ of an induced subgraph $G|_W$ for some $W\se V(G)$, and let us compute the coefficient of its term $\kaugm{\lambda(\mathcal{C}^*)}$ on the right-hand side. By Lemma~\ref{lem:tc_identity}, the term indexed by $C$ produces $\mathcal{C}^*$ in its expansion exactly when $W=V'\cup S$ for some $S\se C$, which is to say that the set $V_H(\mathcal{C}^*)$ of $H$-vertices that $\mathcal{C}^*$ covers is contained in $C$. So the total coefficient of $\mathcal{C}^*$ is $$ \sum_{C:\,V_H(\mathcal{C}^*)\se C\se V(H)}(-1)^{|V(H)\setminus C|}. $$ To evaluate this, note that choosing $C$ amounts to deciding which of the $|V(H)\setminus V_H(\mathcal{C}^*)|$ remaining $H$-vertices to include, so by the binomial theorem the sum equals $(1-1)^{|V(H)\setminus V_H(\mathcal{C}^*)|}$. This is $1$ when $V_H(\mathcal{C}^*)=V(H)$ and $0$ otherwise. In other words, the only covers that remain are exactly those covering all of $V'$ and all of $V(H)$, hence all of $G$, and each remains with coefficient $1$. That is precisely the generating function $\ksf{G}$.
\end{proof}

A small manipulation of the formula in Theorem~\ref{thm:ksf-general} gives an alternative expression.

\begin{cor} \label{cor:ksf-alternative}
    Let $G$ be as in Theorem~\ref{thm:ksf-general}, with $Y_{H|_C}$ as in Definition~\ref{def:Y}. The KSF of $G$ is also given by $$ \ksf{G}=\sum_{C\se V(H)}(-1)^{|V(H)\setminus C|}\left[\left(\bigodot_{i=1}^d\sum_{S\se C}\ksf{G_i|_{V_i'\cup S}}\right)\odot(Y_{H|_C})^{\odot(-d+1)}\right], $$ where the negative exponent $-d+1$ is interpreted using the $\odot$-inverse of $Y_{H|_C}$.
\end{cor}

\begin{proof}
    This is just a regrouping of the join products inside Theorem~\ref{thm:ksf-general}. The point is that the $d$ copies of $Y_{H|_C}^{\odot(-1)}$, one for each $i$, can be pulled out and combined with the single remaining $Y_{H|_C}$ into one power. Using the fact that $\odot$ is commutative and associative, we can rewrite the summand of Theorem~\ref{thm:ksf-general} as
\begin{align*}
\left(\bigodot_{i=1}^d\frac{\sum_{S\se C}\ksf{G_i|_{V_i'\cup S}}}{Y_{H|_C}}\right)\odot Y_{H|_C} &=\left(\bigodot_{i=1}^d\left(\sum_{S\se C}\ksf{G_i|_{V_i'\cup S}}\right)\right) \odot\left(\bigodot_{i=1}^d Y_{H|_C}^{\odot(-1)}\right)\odot Y_{H|_C}\\ &=\left(\bigodot_{i=1}^d\sum_{S\se C}\ksf{G_i|_{V_i'\cup S}}\right) \odot Y_{H|_C}^{\odot(-d)}\odot Y_{H|_C}^{\odot(1)}\\ &=\left(\bigodot_{i=1}^d\sum_{S\se C}\ksf{G_i|_{V_i'\cup S}}\right) \odot Y_{H|_C}^{\odot(-d+1)}.
\end{align*}
Substituting this back into the sum over $C$ gives the stated formula.
\end{proof}


We give an example now to show what the formula looks like in the simplest case, where the shared subgraph is a single vertex, and to illustrate how the same answer can be reached directly. This is the Kromatic analogue of Example~\ref{prop:single_vertex_join}.

\begin{example}
    Let $G_1$ and $G_2$ be graphs such that $V(G_1) \cap V(G_2) = \{v\}$, and let $G$ be the graph formed by the union of $G_1$ and $G_2$ along with all edges connecting $V(G_1)\setminus\{v\}$ to $V(G_2)\setminus\{v\}$. Then $$ \ksf{G} = \left(\ksf{G_1} + \ksf{G_1 \setminus v}\right) \odot \left(\ksf{G_2} + \ksf{G_2 \setminus v}\right) \odot (1+\kaugm{1})^{\odot (-1)} - \ksf{G_1 \setminus v} \odot \ksf{G_2 \setminus v}. $$ It can also be derived directly by inclusion-exclusion on how $v$ is covered. Let $f(v, G_i)$ be the generating function for stable set covers of $G_i$ in which $v$ is not covered as a singleton $\{v\}$. We have $f(v, G_i) = (\ksf{G_i} - \ksf{G_i \setminus v} \odot \kaugm{1}) \odot (1+\kaugm{1})^{\odot (-1)}$, and rearranging gives $$ f(v,G_i)+\ksf{G_i \setminus v}=(\ksf{G_i} + \ksf{G_i \setminus v}) \odot (1+\kaugm{1})^{\odot (-1)}, $$ which is the generating function for the collections of stable sets of $G_i$ that do not use the singleton $\{v\}$ and cover every vertex of $G_i$ except possibly $v$. Since every vertex of $V(G_1)\setminus\{v\}$ is adjacent to every vertex of $V(G_2)\setminus\{v\}$, no stable set of $G$ other than $\{v\}$ can reach into both sides at once, so taking the $\odot$-product of these components for $G_1$ and $G_2$ and multiplying by $(1+\kaugm{1})$ to account for the singleton $\{v\}$ builds each stable set cover of $G$ exactly once, along with the collections that leave $v$ uncovered. Those are the ones where no stable set contains $v$ and the singleton is not used, which are counted by $\ksf{G_1 \setminus v} \odot \ksf{G_2 \setminus v}$, so subtracting them gives the desired result.
\end{example}

\subsection{Disjoint-type constructions}\label{subsec:disjoint-type}

We now turn to the second construction. Here the graphs $G_1,\dots,G_d$ are again identified along the common subgraph $H$, but no extra edges are added between the sets $V_i'$. Unlike the formulas of \S\ref{subsec:join-type}, the right-hand sides of Theorems~\ref{thm:csf-disjoint} and~\ref{thm:ksf-disjoint} below contain $\csf{G}$ or $\ksf{G}$ itself with coefficient 1 (from the term with $U=S=\varnothing$, or from the terms with $W=C=V'$ and $S=V(G)$), so they are identities saying that all the other terms add up to 0, rather than recursive formulas. In them, $\prod$ denotes the ordinary product, while $\odot$ and $\odot$-inverses are as in \S\ref{sec:background}.

\begin{theorem} \label{thm:csf-disjoint}
    Let $G_1,\dots,G_d$ be graphs sharing a common induced subgraph $H$, so that $G_i|_{V(H)}=H$ for all $i$ and $V(G_i)\cap V(G_j)=V(H)$ for all $i\ne j$, and write $V_i':=V(G_i)\setminus V(H)$ and $V':=\bigcup_{i=1}^d V_i'$. Let $G$ be the graph on the vertex set $\bigcup_{i=1}^d V(G_i)$ whose edges are exactly the edges of the $G_i$, so that no edge joins $V_i'$ to $V_j'$ for $i\ne j$. The chromatic symmetric function of $G$ is given by
    \begin{equation}
        \csf{G}=\sum_{U\se V'}\left[\left(\prod_{i=1}^d\csf{G|_{U\cap V_i'}}\right)\odot\left(\sum_{S\se V'\setminus U}\left(\sum_{\lambda}(-1)^{\ell(\lambda)}\augm{\lambda}\right)\odot\csf{G|_{V(H)\cup(V'\setminus(U\cup S))}}\right)\right],
    \end{equation}
    where the innermost sum is taken over all stable set partitions of $G|_S$, with $\lambda\vdash|S|$ being the integer partition corresponding to each such set partition. Stable set partitions with the same block sizes are counted separately, and for $S=\varnothing$ the sum has value $1$.
\end{theorem}

\begin{proof}
    We sort the vertices of $V'$ according to whether their stable set stays inside $V'$ or meets the shared subgraph $H$. Fix $U\se V'$, the set of vertices of the first kind. No stable set of such a partition meets both $U$ and $V(H)\cup(V'\setminus U)$, so the contributions of the two parts combine under $\odot$.

    For the first part, because the sets $V_i'$ are pairwise disjoint with no edges between them, $G|_U$ is just the disjoint union of the $G|_{U\cap V_i'}$. The CSF of a disjoint union is the product of the CSFs, so this part contributes $\prod_{i=1}^d\csf{G|_{U\cap V_i'}}$.

    For the second part, we want exactly the stable set partitions of $G|_{V(H)\cup(V'\setminus U)}$ in which every stable set meets $V(H)$. By inclusion-exclusion over the stable sets contained in $V'\setminus U$, their generating function is
    \begin{equation*}
        \sum_{S\se V'\setminus U}\left(\sum_{\lambda}(-1)^{\ell(\lambda)}\augm{\lambda}\right)\odot\csf{G|_{V(H)\cup(V'\setminus(U\cup S))}}.
    \end{equation*}
    To see this, fix a stable set partition of $G|_{V(H)\cup(V'\setminus U)}$. Each term producing this partition chooses some of its stable sets contained in $V'\setminus U$, takes $S$ to be their union, and leaves all other stable sets in the remaining CSF. The chosen stable sets form the stable set partition of $G|_S$ in the innermost sum, so each chosen stable set contributes a factor of $-1$.

    If the fixed partition has a stable set contained in $V'\setminus U$, choose one such stable set. Pair each choice of stable sets with the choice obtained by adding or removing this stable set. The two choices produce the same stable set partition with opposite signs, so their contributions cancel. If the fixed partition has no stable set contained in $V'\setminus U$, only the empty choice contributes, with coefficient $1$. Thus the alternating sum leaves exactly the partitions in which every stable set meets $V(H)$.

    Joining the two parts with $\odot$ therefore counts exactly the stable set partitions of $G$ for which $U$ is the union of the stable sets disjoint from $V(H)$. Every stable set partition of $G$ determines a unique such $U$, so summing over all $U\se V'$ gives the stated formula.
\end{proof}

We now solve the Kromatic analogue of the same problem, starting with two preparatory lemmas.

\begin{lemma}
    \label{lem:disjoint_factor1}
    Let $G$, $V_1',\dots,V_d'$ and $V'$ be as in Theorem~\ref{thm:csf-disjoint}, and let $W\subseteq V'$. The generating function for stable set covers of the vertices in $V'\setminus W$ by stable sets contained entirely within $V'$ is given by $$ \prod_{i=1}^d\sum_{V_i'\setminus W\subseteq S\subseteq V_i'}\ksf{G|_S}. $$
\end{lemma}

\begin{proof}
    This again uses the disjointness of the sets $V_i'$, but at the level of the product rather than of the individual stable sets. Because there are no edges between the different $V_i'$, for any $S\se V'$ the graph $G|_S$ is the disjoint union of the graphs $G|_{S\cap V_i'}$, and the KSF of a disjoint union is the product of the KSFs. Moreover, the sets $S$ with $V'\setminus W\se S\se V'$ are exactly those obtained by choosing each $S\cap V_i'$ independently with $V_i'\setminus W\se S\cap V_i'\se V_i'$. Expanding the product therefore gives $$ \prod_{i=1}^d\sum_{V_i'\setminus W\se S\se V_i'}\ksf{G|_S}=\sum_{V'\setminus W\se S\se V'}\ksf{G|_S}. $$

    It remains to identify the right hand side. A collection of stable sets contained in $V'$ that covers $V'\setminus W$ is the same thing as a stable set cover of $G|_S$, where $S$ is the set of vertices the collection covers, and the sets $S$ that can arise this way are exactly those with $V'\setminus W\se S\se V'$. So the right hand side is the stated generating function.
\end{proof}

That takes care of the stable sets that stay inside $V'$. Next we count the covers on the other side, where the condition to impose is that every stable set used reaches into the shared subgraph:

\begin{lemma}
    \label{lem:disjoint_factor2}
    Let $G$, $H$, $V_1',\dots,V_d'$ and $V'$ be as in Theorem~\ref{thm:csf-disjoint}, and let $W\subseteq V'$. The generating function for stable set covers of $V(H)\cup W$ in which every stable set intersects $V(H)$ is given by $$ \sum_{C\subseteq W}(-1)^{|W\setminus C|}\left(\sum_{V(H)\subseteq S\subseteq V(H)\cup C}\ksf{G|_S}\right)\odot Y_{G|_C}^{\odot(-1)}. $$
\end{lemma}

\begin{proof}
   The trick is to use the ``division'' to discard the stable sets lying entirely outside $V(H)$, and then to impose, by inclusion-exclusion on $W$, the condition that all of $W$ gets covered.

   For each $C\se W$, observe the covers of induced subgraphs $G|_S$ with $V(H)\se S\se V(H)\cup C$, all of which cover $V(H)$. Any such cover factors uniquely into the stable sets that lie inside $C$ and the stable sets that meet $V(H)$, so summing over all admissible $S$, the value $\sum_{V(H)\se S\se V(H)\cup C}\ksf{G|_S}$ is $Y_{G|_C}$ join-producted with the generating function for the collections whose stable sets all meet $V(H)$ and which cover $V(H)$ together with some subset of $C$, and dividing out $Y_{G|_C}$ isolates that second factor. Finally, such a collection is counted in the term indexed by $C$ exactly when $C$ contains the vertices of $W$ that the collection covers, so the alternating sum over $C\se W$ removes the collections that fail to cover all of $W$, which yields the claimed formula.
\end{proof}

Sorting the stable sets of a cover of $G$ by whether or not they meet the shared subgraph, the two lemmas count the two halves of the cover independently, so we can join them together and sum over the choices of $W$:

\begin{theorem} \label{thm:ksf-disjoint}
    Let $G$, $H$, $V_1',\dots,V_d'$ and $V'$ be as in Theorem~\ref{thm:csf-disjoint}. With $Y_K$ as in Definition~\ref{def:Y}, the Kromatic symmetric function of $G$ is given by $$ \ksf{G}=\sum_{W\subseteq V'}\left[\left(\prod_{i=1}^d\sum_{V_i'\setminus W\subseteq S\subseteq V_i'}\ksf{G|_S}\right)\odot\left(\sum_{C\subseteq W}(-1)^{|W\setminus C|}\left(\sum_{V(H)\subseteq S\subseteq V(H)\cup C}\ksf{G|_S}\right)\odot Y_{G|_C}^{\odot(-1)}\right)\right]. $$
\end{theorem}

\begin{proof}
    This mirrors the CSF proof, except that we split the stable sets of a cover rather than the vertices. Fix a stable set cover of $G$, and let $W\se V'$ be the set of vertices outside $H$ that belong to stable sets of the cover meeting $V(H)$. Every stable set of the cover either meets $V(H)$ or is contained in $V'$, so the cover splits in exactly one way into the collection of its stable sets meeting $V(H)$, which is a stable set cover of $V(H)\cup W$ in which every stable set meets $V(H)$, and the collection of its stable sets contained in $V'$, which must cover $V'\setminus W$ since no other stable set of the cover reaches those vertices. Conversely, for a fixed $W\se V'$, joining a collection of the first kind with a collection of the second kind gives a stable set cover of $G$, and both $W$ and the two collections can be read back off from it, since a stable set belongs to the first collection exactly when it meets $V(H)$. The two kinds of collections are exactly what Lemmas~\ref{lem:disjoint_factor2} and~\ref{lem:disjoint_factor1} enumerate, so the generating function for a fixed $W$ is the $\odot$ product of the two parts. Summing over all $W\se V'$ then enumerates every stable set cover of $G$ once, which is $\ksf{G}$.
\end{proof}

\section{An inclusion-exclusion based power sum expansion}\label{sec:power-sum}


Stanley gives two expansions of $\csf{G}$ in the classical power sum basis $\{p_\lam\}$. The first is an inclusion-exclusion formula over the edge set of $G=(V,E)$ \cite[Theorem 2.5]{stanley1995symmetric}: $$ \csf{G}=\sum_{S \se E} (-1)^{|S|} p_{\lam(S)}, $$ where $\lam(S)$ is the partition of $|V|$ given by the sizes of the connected components of $G_S$. The second is a cancellation-free expansion over the graph's bond lattice \cite[Theorem 2.6]{stanley1995symmetric}. A partition $\pi$ of $V$ is a \newword{connected partition} if each of its blocks induces a connected subgraph of $G$, and the connected partitions form a lattice $L_G$ called the \newword{bond lattice} or \newword{lattice of contractions}. Writing $\lam(\pi)$ for the partition whose parts are the block sizes of $\pi$, Stanley's second expansion is $$ \csf{G}=\sum_{\pi \in L_G} \mu(\hat{0}, \pi) p_{\lam(\pi)}, $$ where $\mu$ is the Möbius function of $L_G$ and $\hat{0}$ denotes the minimum element of $L_G$, namely the partition of $V$ into singleton blocks.

We now give a $K$-analogue of these expansions, addressing a question from \cite[Remark 3.5]{crew2023kromatic} about finding an inclusion-exclusion version of the power sum expansion for the KSF. An expansion of $\ksf{G}$ in the $K$-theoretic power sum basis $\{\ol{p}_\lam\}$ already exists \cite{pierson2025lyndon}; our point here is that an inclusion-exclusion description of the coefficients explains why they recover Stanley's bond lattice expansion in the lowest degree.

Recall that the classical \newword{power sum symmetric functions} form a multiplicative basis for $\Lam$, with $p_\lam:=p_{\lam_1}\cdots p_{\lam_{\ell(\lam)}}$ and $p_n:=\sum_{i\ge 1}x_i^n$. The \newword{$K$-theoretic power sum basis} $\{\ol{p}_\lam\}$ is also multiplicative, with $\ol{p}_\lam:=\ol{p}_{\lam_1}\cdots\ol{p}_{\lam_{\ell(\lam)}}$, where $\ol{p}_k$ is defined by $1+\ol{p}_k=\prod_{i\ge 1}(1+x_i^k)$ \cite{crew2023kromatic}.

The terms of our expansion are indexed by the following method. We consider finite sequences $\mathbf{V}=(V_1,\dots,V_l)$ of nonempty multisets of vertices of $G$ such that every vertex of $G$ lies in at least one $V_i$ and the vertices of each $V_i$ induce a connected subgraph of $G$, so that each such $\mathbf{V}$ covers $G$ by connected pieces. Only the multiset of parts will matter, so we always list the parts in a canonical order, first by decreasing multiset size and then lexicographically with respect to a fixed ordering of the vertices, and we identify two sequences that agree after this reordering. The symmetric function contributed by such a cover is the product $\ol{p}_{|V_1|}\cdots\ol{p}_{|V_l|}$, and we measure the total overlap among its parts by the quantity $$ \sum_{i=1}^l(|V_i|-1), $$ which is 0 exactly when every $V_i$ consists of a single vertex. The total overlap is what lets us order these covers so that the coefficients can be computed one at a time, since the covers we will need to already know about when computing a given coefficient all turn out to have smaller total overlap, or the same total overlap and fewer parts.

The idea is to set up the coefficients so that for any single coloring, the contributions from all the terms add up to 1 when the coloring is a proper set coloring and 0 when it is not, where here a coloring may give each vertex any nonempty finite multiset of colors. To make sense of this, we first need to say how a power sum term produces colorings. The product $\ol{p}_{|V_1|}\cdots\ol{p}_{|V_l|}$ counts a choice of a nonempty color set for each part $V_i$, made independently, and a vertex then receives each color chosen for a part containing it once for each time it appears in that part. We say that a coloring $\kappa$ is \newword{generated by} $\mathbf{V}$ if it can arise this way, and we call such a choice a \newword{generation} of $\kappa$ from $\mathbf{V}$, so that a generation is a way of choosing a nonempty color set for each part with every vertex $v$ ending up with exactly the multiset of colors $\kappa(v)$. Each cover also carries a coloring of its own, the one giving every vertex $v$ each index $i$ as many times as $v$ appears in $V_i$, so that each part contributes its own color; we call this the \newword{associated coloring} of the cover.

We can now build the coefficients $c_{\mathbf{V}}$, working through the covers in order of increasing total overlap, and in order of increasing number of parts among covers with the same total overlap.
\begin{algorithm}\label{alg:coefficients}
Proceed through the following steps:
    \begin{enumerate}
    \item The covers of total overlap 0 are exactly those whose parts are all single vertices. For the cover in which each vertex of $G$ appears exactly once, which corresponds to the term $\ol{p}_1^{|V|}$ and generates every set coloring of $V$, whether or not it is proper, set $c_{\mathbf{V}}=1$. For every other cover of total overlap 0, set $c_{\mathbf{V}}=0$.
    \item For a cover $\mathbf{V}$ of positive total overlap, set $c_{\mathbf{V}}$ equal to $-\sum c_{\mathbf{V}'}$, divided by the product of the factorials of the multiplicities of the distinct parts of $\mathbf{V}$. Here the sum is taken over all pairs consisting of a cover $\mathbf{V}'$ with either smaller total overlap than $\mathbf{V}$, or the same total overlap and fewer parts, together with a generation of the associated coloring of $\mathbf{V}$ from $\mathbf{V}'$. In particular, each such $\mathbf{V}'$ contributes once for each generation of this coloring from $\mathbf{V}'$, and its coefficient $c_{\mathbf{V}'}$ has already been determined.
\end{enumerate}
\end{algorithm}
In terms of these coefficients, our expansion is as follows.

\begin{theorem}
\label{thm:pie_expansion}
For any graph $G=(V,E)$, the Kromatic symmetric function of $G$ can be written as the infinite sum $$ \ksf{G}=\sum_{\mathbf{V}} c_{\mathbf{V}}\,\ol{p}_{|V_1|}\cdots\ol{p}_{|V_l|}, $$ where $\mathbf{V}=(V_1,\dots,V_l)$ ranges over all sequences of connected multisets covering $G$ as above and the rational numbers $c_{\mathbf{V}}$ are the coefficients computed above.
\end{theorem}

The sum is infinite, but it is well defined degree by degree which makes sense because the product $\ol{p}_{|V_1|}\cdots\ol{p}_{|V_l|}$ has no terms of degree below $|V_1|+\dots+|V_l|$, and for any given $d$ there are only finitely many covers with $|V_1|+\dots+|V_l|\le d$, so only finitely many terms contribute in each degree.

To actually run this recursion, we need to know which generations appear in it, and these turn out to have a simple description that does not mention colorings at all:

\begin{prop}
\label{prop:comb_interp_N}
Let $\mathbf{V}=(V_1,\dots,V_l)$ and $\mathbf{W}=(W_1,\dots,W_m)$ be two covers as above. Then the generations of the associated coloring of $\mathbf{W}$ from $\mathbf{V}$ correspond bijectively to the ways to write each part of $\mathbf{W}$ as a union of some of the parts of $\mathbf{V}$, with multiplicities added, using every part of $\mathbf{V}$ at least once.
\end{prop}

\begin{proof}
The associated coloring of $\mathbf{W}$ uses the color set $[m]$, giving each vertex $v$ each color $i$ as many times as $v$ appears in $W_i$. A generation of it from $\mathbf{V}$ assigns a nonempty color set $\mathcal{C}_j\se[m]$ to each $V_j$ so that each vertex $v$ picks up exactly those colors, with those multiplicities. Fix a color $i$. Each vertex $v$ picks up $i$ once for each time it appears in a part $V_j$ with $i\in\mathcal{C}_j$, so the number of times $v$ appears in $W_i$ is the total number of times it appears in these parts. Thus $W_i$ is the union of the parts $V_j$ with $i\in\mathcal{C}_j$, with multiplicities added, and since each $\mathcal{C}_j$ is nonempty, every part of $\mathbf{V}$ gets used at least once. Conversely, any such collection of ways to write the $W_i$ determines the sets $\mathcal{C}_j$ and hence the generation, so the correspondence is a bijection, as claimed.
\end{proof}

It is helpful to picture this recursion as taking place on a directed multigraph whose vertices are the covers $\mathbf{V}$ of $G$ by connected multisets, with one directed edge from $\mathbf{V}'$ to $\mathbf{V}$ for each generation of the associated coloring of $\mathbf{V}$ from $\mathbf{V}'$, for $\mathbf{V}'\ne\mathbf{V}$. Ordering by total overlap and then by number of parts makes this graph acyclic, and for $\mathbf{V}$ of positive total overlap, $c_{\mathbf{V}}$ is minus the sum of the previously computed coefficients over all incoming edges, divided by the product of the factorials of the multiplicities of the distinct parts of $\mathbf{V}$. This is much like a Möbius inversion on a poset, generalized here to a directed multigraph.

\begin{proof}[Proof of Theorem~\ref{thm:pie_expansion}]
Expanding each product $\ol{p}_{|V_1|}\cdots\ol{p}_{|V_l|}$ as a sum over the colorings it generates, counted once for every generation, the proposed sum becomes $$ \sum_{\mathbf{V}}\ \sum_{\kappa}\ \sum_{\substack{\tn{generations of}\\\kappa\tn{ from }\mathbf{V}}}c_{\mathbf{V}}\prod_{v\in V}\prod_{i\in\kappa(v)}x_i, $$ where the middle sum is over all colorings of $V$ and each color in $\kappa(v)$ is counted with its multiplicity. So it suffices to show that for any coloring $\kappa$, adding up $c_{\mathbf{V}}$ over all pairs consisting of a cover $\mathbf{V}$ and a generation of $\kappa$ from $\mathbf{V}$ gives 1 if $\kappa$ is a proper set coloring and 0 otherwise.

Given $\kappa$, let $\mathbf{W}$ be the sequence of multisets given by the connected components of the color classes of $\kappa$, where the color class of a color is the multiset of vertices using it, with each vertex appearing as many times as it uses that color, and each component keeps these multiplicities. Then $\mathbf{W}$ is a cover of the kind considered above by construction, and since the parts of such a cover are always connected, any generation of $\kappa$ has to respect those components, so for every $\mathbf{V}$ the generations of $\kappa$ from $\mathbf{V}$ match the generations of the associated coloring of $\mathbf{W}$ from $\mathbf{V}$, using also that generations do not depend on the names of the colors. Moreover, $\kappa$ is a proper set coloring exactly when every color class is a stable set, which happens exactly when every component of every color class is a stable set, which is exactly the condition for the associated coloring of $\mathbf{W}$ to be a proper set coloring. So it is enough to prove the claim for that coloring in place of $\kappa$.

By Proposition \ref{prop:comb_interp_N}, if a cover $\mathbf{V}=(V_1,\dots,V_l)$ generates the associated coloring of $\mathbf{W}=(W_1,\dots,W_m)$, fix such a generation, let $\mc{C}_j\se[m]$ be the nonempty color set it assigns to $V_j$ as in the proof of that proposition, and let $J_i=\{j\in[l]\mid i\in\mc{C}_j\}$, which is nonempty since $W_i$ is. Then $|W_i|=\sum_{j\in J_i}|V_j|$, and since $\sum_{i=1}^m|J_i|=\sum_{j=1}^l|\mc{C}_j|$, we get $$ \sum_{i=1}^m(|W_i|-1)-\sum_{j=1}^l(|V_j|-1)=\sum_{i=1}^m\bigl(|J_i|-1\bigr)+\sum_{j=1}^l(|V_j|-1)\bigl(|\mc{C}_j|-1\bigr), $$ where every term on the right is nonnegative. So the total overlap of $\mathbf{V}$ is at most that of $\mathbf{W}$, with equality only if each part of $\mathbf{W}$ is a single part of $\mathbf{V}$. In that case, since every part of $\mathbf{V}$ is used, $\mathbf{V}$ has at most as many parts as $\mathbf{W}$, with equality only if $\mathbf{V}=\mathbf{W}$, in which case a generation just matches each part of $\mathbf{W}$ with an equal part of $\mathbf{W}$. So the sum we are computing is $c_\mathbf{W}$ times the product of the factorials of the multiplicities of the distinct parts of $\mathbf{W}$, plus the sum of $c_{\mathbf{V}}$ over all pairs consisting of a cover $\mathbf{V}$ that has smaller total overlap than $\mathbf{W}$, or the same total overlap and fewer parts, and a generation of the associated coloring of $\mathbf{W}$ from $\mathbf{V}$. If $\mathbf{W}$ has positive total overlap, then some part of $\mathbf{W}$ has size more than 1, so it either contains some vertex more than once or, being connected, contains two adjacent vertices, and thus that coloring is not a proper set coloring, while by the recursion defining the coefficients, the sum is exactly 0, as needed. If $\mathbf{W}$ has total overlap 0, then that coloring is a proper set coloring, and the only cover with nonzero coefficient that generates it is the one using each vertex exactly once, which does so in exactly one way, so the sum is 1, as needed.
\end{proof}

The same machinery works for the alternative $K$-theoretic power sum basis $\{\ol{p}'_\lam\}$, defined in \cite{pierson2025lyndon} by $$ 1+\ol{p}'_k=\prod_{i\ge 1}\frac1{1-x_i^k}=\prod_{i\ge 1}(1+x_i^k+x_i^{2k}+\cdots). $$ The only difference is that a color can now be assigned to a part more than once, so that a generation of the associated coloring of $\mathbf{W}$ from $\mathbf{V}$ now writes each part of $\mathbf{W}$ as a union of parts of $\mathbf{V}$ in which the same part of $\mathbf{V}$ may be used several times, and we write $c'_{\mathbf{V}}$ for the coefficients given by the same recursion with respect to these more general generations.

\begin{theorem}
\label{thm:pie_expansion_p_prime}
The Kromatic symmetric function of $G$ expands in the $\ol{p}'$ basis as $$ \ksf{G}=\sum_{\mathbf{V}}c'_{\mathbf{V}}\,\ol{p}'_{|V_1|}\cdots\ol{p}'_{|V_l|}, $$ where $\mathbf{V}=(V_1,\dots,V_l)$ again ranges over the same covers of $G$ as in Theorem~\ref{thm:pie_expansion}.
\end{theorem}

\begin{proof}
The argument is identical to that of Theorem~\ref{thm:pie_expansion}, with the single change that a color may now be assigned to a part more than once, which is exactly what the more general generations account for (in the proof of that theorem, $\mc{C}_j$ is then a multiset, and $J_i$ contains each $j$ as many times as $i$ appears in $\mc{C}_j$).
\end{proof}

We use this $K$-theoretic framework to give an alternative proof of Stanley's power sum expansion for the chromatic symmetric function.

\begin{cor}[Stanley {\cite[Theorem 2.6]{stanley1995symmetric}}]
\label{cor:csf_p_expansion_bond}
The expansion of the chromatic symmetric function $\csf{G}$ in the power sum basis is $$  \csf{G}=\sum_{\pi \in L_G} \mu(\hat{0}, \pi) p_{\lambda(\pi)}.  $$
\end{cor}
\begin{proof}
By Remark \ref{rem:ksf_to_csf}, we can recover $\csf{G}$ by taking the lowest degree terms in Theorem \ref{thm:pie_expansion}. Since $1+\ol{p}_k=\prod_{i\ge 1}(1+x_i^k)$, the lowest degree part of $\ol{p}_k$ is $p_k$, so the lowest degree part of the product $\ol{p}_{|V_1|}\cdots\ol{p}_{|V_l|}$ is $p_{|V_1|}\cdots p_{|V_l|}$, in degree $|V_1|+\dots+|V_l|$. Now for any of our covers, $|V_1|+\dots+|V_l|\ge|V|$ because the supports cover $V$, with equality exactly when the multisets are pairwise disjoint sets, and since their supports are connected, this happens exactly when $\mathbf{V}$ lists the blocks of a connected partition of $V$. So the degree $|V|$ terms are precisely those indexed by the elements of $L_G$, and identifying each $\pi\in L_G$ with the cover given by its blocks, we get $$ \csf{G}=\sum_{\pi\in L_G}c_\pi\, p_{\lam(\pi)}. $$ Thus, it suffices to show that $c_\pi=\mu(\hat{0},\pi)$.

We can do this by induction on the total overlap of the cover given by the blocks of $\pi$. Note that this total overlap is $|V|-\ell(\lam(\pi))$, so ordering covers by increasing total overlap orders the connected partitions by decreasing number of blocks, and the cover by singletons used in the base case of the recursion is just $\hat{0}$, giving $c_{\hat{0}}=1=\mu(\hat{0},\hat{0})$ for the base case. For any other $\pi$, we just need to know which covers generate the associated coloring of the cover given by the blocks of $\pi$. In such a generation, each part receives a nonempty set of colors, and the colors of each vertex get split up among the parts containing it. Since each vertex has only one color, each part must receive a single color and thus lie inside the corresponding block of $\pi$, and each vertex must lie in exactly one part, with multiplicity 1. So the covers generating it are exactly the connected partitions refining $\pi$, each in exactly one way, and those other than $\pi$ itself are exactly $\sigma<\pi$. The recursion is thus $$ c_\pi=-\sum_{\sigma<\pi}c_\sigma=-\sum_{\sigma<\pi}\mu(\hat{0},\sigma)=\mu(\hat{0},\pi), $$ using the inductive hypothesis and then the recursion of the Möbius function of $L_G$, as needed.
\end{proof}


\section{Hopf algebra interpretations}\label{sec:hopf}

In this final section, our goal is to reinterpret both the $\ol{p}$-expansion from \cite{pierson2025lyndon} and the $\m$-expansion from \cite{crew2023kromatic} using Marberg's Hopf algebra $\mWGraphs$ from \cite{marberg2023kromatic}. We will also develop some other combinatorially interesting formulas along the way.

\subsection{Background on Hopf algebras}\label{subsec:hopf-background}

We follow Marberg \cite{marberg2023kromatic}. Previously, our graphs have mostly been unweighted, but for the Hopf algebra theory it is convenient to work with weighted graphs, which reduce to the unweighted setting when all vertex weights $w_1,\dots,w_n$ equal $1$.

A \newword{Hopf algebra} is an algebra $A$ over some field $\mb{K}$, which we always take to have characteristic 0 so that the exponential and logarithm series of \S\ref{subsec:grouplike} are defined, together with a \newword{coproduct} $\Delta:A\to A\otimes A$, and a \newword{counit} $\epsilon: A\to \mb{K}$ that are both algebra homomorphisms, satisfying certain compatibility axioms. In particular, there must be an \newword{antipode} $S:A\to A$, which is a linear map (and in our cases an algebra homomorphism) satisfying the property that if $\Delta f = \sum g\otimes h,$ then $$ \sum S(g)h = \sum gS(h) = \epsilon(f)1_A. $$ A Hopf algebra is \newword{commutative} if the multiplication is commutative, and \newword{cocommutative} if swapping the order of the factors in each term of the coproduct keeps the coproduct the same. A Hopf algebra is \newword{graded} if it can be broken down as $A = \bigoplus_{n\ge 0}A_n$ such that if $f\in A_m$ and $g\in A_n$, then $fg\in A_{m+n}$, and if $f\in A_n$ and $\Delta f = \sum g\otimes h$, then for each $g\in A_k$, we have $h\in A_{n-k}$. A graded Hopf algebra is \newword{connected} if $A_0\cong \mb{K}$.

A \newword{character} is an algebra homomorphism $\zeta:A\to\mb{K}$, and a \newword{combinatorial Hopf algebra (CHA)} $(A,\zeta_A)$ is a graded connected Hopf algebra together with a specific choice of character $\zeta_A$. The coproduct, counit, and antipode put a group structure on the characters of $A$, with multiplication given by $$ (\zeta_1*\zeta_2)(f) = \sum \zeta_1(g)\zeta_2(h), $$ where again $\Delta f = \sum g\otimes h$, and multiplicative inverses $$ \zeta^{-1}(f) = \zeta(S(f)). $$ The compatibility axioms guarantee that this is a valid group structure.

The \newword{Hopf algebra of weighted graphs} $\textsf{WGraphs}$ is freely generated by the set of all connected graphs up to isomorphism, with product given by disjoint unions $G\sqcup H$, counit sending the empty graph to 1 and all other graphs to 0, and coproduct given by $$ \Delta G := \sum_{T\sqcup U = V(G)}G|_T\otimes G|_U. $$ $\WGraphs$ can be turned into a graded Hopf algebra by saying that $G\in \WGraphs_n$ if the sum of the weights of the vertices of $G$ is $n$, in which case $\WGraphs$ is also connected. We can also see from the definition that $\WGraphs$ is commutative and cocommutative.

The connection between the CSF and Hopf algebras is that if one makes $\textsf{WGraphs}$ into a CHA by equipping it with the character $\zeta_\textsf{WGraphs}:\textsf{WGraphs}\to \mb{K}$ given by 
$$ \zeta_\WGraphs(G) := \begin{cases}
    1 &\tn{if }G\tn{ is edgeless},\\
    0 & \tn{else,}
\end{cases} $$ 
then there is a unique induced CHA map $(\WGraphs,\zeta_\WGraphs) \to (\Sym,\zeta_\Sym)$ given by $G\mapsto X_G$, where $\Sym$ is the \newword{Hopf algebra of symmetric functions}, which is the ring $\Lam$ from \S\ref{sec:background} with coefficients in $\mb{K}$, with coproduct given by $$ \Delta p_n := p_n \otimes 1 + 1 \otimes p_n, $$ where $p_n:=x_1^n+x_2^n+\dots$ is the $n^{\text{th}}$ \newword{power sum symmetric function}, and character $\zeta_\Sym$ given by specializing one variable $x_1$ to 1 and all other variables to 0. This map is unique because $(\Sym,\zeta_\Sym)$ is a \newword{terminal object} in the category of cocommutative CHAs, meaning there is a unique map that preserves the Hopf algebra structure and commutes with the character maps from any cocommutative CHA to $(\Sym,\zeta_\Sym)$.

Marberg \cite{marberg2023kromatic} defines a modified Hopf algebra of weighted graphs $\mWGraphs$ by first taking the \newword{completion} (meaning allowing infinite linear combinations of graphs) and then keeping the product as being the disjoint union, but modifying the coproduct to be $$ \blacktriangle G := \sum_{T\cup U = V(G)} G|_T\otimes G|_U, $$ where the difference is that the vertex sets $T$ and $U$ are allowed to overlap. Then $\mWGraphs$ is commutative and cocommutative in the same way as $\WGraphs$. Marberg makes $\mWGraphs$ into what he calls a \newword{linearly compact (LC)-Hopf algebra} (essentially a Hopf algebra in which infinite linear combinations are allowed), and equips it with a character $\ol{\zeta}_{\WGraphs}:\mWGraphs\to \mb{K} \ldb t \rdb $, which, like all characters to $\mb{K} \ldb t \rdb $ below, is required to be continuous and to agree with the counit at $t=0$, given by 
$$ \ol{\zeta}_{\WGraphs}(G) := \begin{cases}
    t^{w(G)}& \tn{if }G\tn{ is edgeless},\\
    0 & \tn{else},
\end{cases} $$
where $w(G)$ is the sum of the weights of all the vertices of $G$. The reason he defines the character as mapping to $\mb{K} \ldb t \rdb $ instead of to $\mb{K}$ is so that it is well-defined for infinite linear combinations of graphs. His motivation for introducing this Hopf algebra $\mWGraphs$ is that it gives a Hopf algebra interpretation for the KSF.

The Hopf algebra interpretation of the KSF is that if one makes the completion $\mSym$, which is the ring $\widehat{\Lam}$ from \S\ref{sec:background} with coefficients in $\mb{K}$, into an LC-Hopf algebra by keeping the product and coproduct the same as for $\Sym$, and equips it with the character $\ol{\zeta}_\Sym$ specializing $x_1$ to $t$ and all other variables to 0, then the unique induced map $(\mWGraphs, \ol{\zeta}_{\WGraphs})\to (\mSym,\ol{\zeta}_\Sym)$ is the map $G\mapsto \ol{X}_G$ sending each graph to its KSF. This map is unique because $(\mSym,\ol{\zeta}_\Sym)$ is a terminal object in the category of cocommutative LC-Hopf algebras with characters.

\subsection{The antipode of \texorpdfstring{$\mWGraphs$}{mWGraphs}}\label{subsec:antipode}

In this section we derive two different formulas for the antipode of $\mWGraphs$, which we will use in later sections and also find to be of combinatorial interest in their own right. We first recall Schmitt's antipode formula for $\WGraphs$:

\begin{theorem}[Schmitt \cite{schmitt1994incidence}, (12.1)]\label{thm:schmitt}
    The antipode for $\WGraphs$ is
    $$ S(G) = \sum_{\pi}(-1)^{|\pi|}G_\pi, $$
    where $\pi$ ranges over all ordered set partitions of $V(G)$ (meaning the parts within the set partition are ordered, while the vertices within each part are not ordered), $|\pi|$ is the number of parts in $\pi$, and $G_\pi$ is the disjoint union of the induced subgraphs formed by restricting $G$ to each part of $\pi$.
\end{theorem}

Both this formula and the one in Theorem \ref{thm:humpert-martin} below are stated in \cite{schmitt1994incidence,humpert2012incidence} for unweighted graphs, but they hold verbatim in $\WGraphs$, since the product and coproduct of $\WGraphs$ are given by the same formulas as in the unweighted case, with each induced subgraph simply inheriting its vertex weights from $G$, so the weights play no role in either proof.

We give the following $K$-analogue of Schmitt's result:

\begin{theorem}\label{thm:schmitt_K}
    The antipode for $\mWGraphs$ is $$ \ol{S}(G) = \sum_{\ol{\pi}}(-1)^{|\ol{\pi}|}G_{\ol{\pi}}, $$ where $\ol{\pi}$ ranges over all finite ordered sequences of nonempty subsets of $V(G)$ (where the subsets may overlap, and the same subset may show up arbitrarily many times) such that each vertex is in at least one of the subsets, $|\ol{\pi}|$ is the total number of subsets, and $G_{\ol{\pi}}$ is again the disjoint union of all the induced subgraphs of $G$ corresponding to the subsets in $\ol{\pi}$, with repetition.
\end{theorem}


\begin{proof}
    We use induction on $|V(G)|$. For the base case of $|V(G)|=0$, $G=\varnothing$ must be the empty graph, and then $\ol{S}(\varnothing)=\varnothing$, which matches the claimed formula, since that formula gives a single term with $\ol{\pi}$ as the empty sequence of length 0, and $G_{\ol{\pi}}=\varnothing$ as the empty graph. For any other graph $G$, the antipode needs to satisfy $$ \sum_{T\cup U = V(G)}\ol{S}(G|_T)\sqcup G|_U = 0, $$ and we can assume inductively the claimed formula holds for all graphs with fewer vertices. For terms with $T=V(G)$, the other subset $U$ is allowed to be any subset of $V(G)$, so we can rearrange to get $$ \ol{S}(G)\sqcup \left(\sum_{U\se V(G)}G|_U\right) = -\sum_{\substack{T\cup U = V(G), \\ T\ne V(G)}}\ol{S}(G|_T)\sqcup G|_U. $$ Now we can ``divide" both sides by the second term on the left side, because the multiplicative inverse of that term is a geometric series, which is legal since we allow infinite linear combinations in $\mWGraphs$: $$ \left(\varnothing+\sum_{\substack{U\se V(G),\\ U\ne \varnothing}}G|_U\right)^{-1}=\sum_{k\ge 0}\left(-\sum_{\substack{U\se V(G),\\ U\ne \varnothing}}G|_U\right)^k, $$ where we are using the geometric series formula $(1+x)^{-1}=\sum_{k\ge 0}(-x)^k$ together with the fact that the empty graph $\varnothing$ is the multiplicative identity of $\mWGraphs$. Thus, we can write $$ \ol{S}(G) = \left(\sum_{\substack{T\cup U=V(G),\\ T\ne V(G)}} \ol{S}(G|_T)\sqcup (-G|_U)\right)\sqcup \sum_{k\ge 0}\left(-\sum_{\substack{U\se V(G),\\ U\ne \varnothing}}G|_U\right)^k. $$ We can now describe a bijection between the terms above and terms of the form $(-1)^{|\ol{\pi}|}G_{\ol{\pi}}$. By the inductive hypothesis, each $\ol{S}(G|_T)$ can be written as the sum of all terms of the form $(-1)^{|\ol{\pi}'|}G_{\ol{\pi}'}$ such that $\ol{\pi}'$ is a sequence of subsets of $V(G)$ that use only vertices in $T$, but use all vertices in $T$ at least once. Thus, in a term of $\ol{S}(G)$, we are starting with such a sequence $\ol{\pi}'$. Then we multiply by $(-G|_U)$, which appends the subset $U$ to the sequence $\ol{\pi}'$. We have increased the length of the sequence by 1 and also flipped the sign of the term, thus maintaining the fact that the sign is equal to $(-1)^{\text{length}}$. Also, since $U$ must include all vertices not in $T$, our sequence now covers all vertices in $V(G)$. Then taking a term from the final sum corresponds to appending any number of additional subsets to our sequence, flipping the sign each time and thus maintaining the fact that the sign is $(-1)^{\text{length}}$. The resulting sequence is the $\ol{\pi}$ corresponding to the term from the product.

    To see that this is indeed a bijection, note that $U$ can be recovered from $\ol{\pi}$ because it is the first term such that if we cut off the sequence after $U$, all vertices are covered. Then $T$ can also be recovered from $\ol{\pi}$ since it is the subset of vertices already covered before that term. Then we need to take the term from $\ol{S}(G|_T)$ corresponding to the initial subsequence of $\ol{\pi}$ ending before $U$, and the term from the final sum corresponding to the final subsequence of $\ol{\pi}$ starting after $U$. This gives us a unique term in the $\ol{S}(G)$ sum corresponding to each $\ol{\pi}$, so the correspondence is a bijection. Thus, by induction, $\ol{S}(G)$ is always equal to the sum of the $(-1)^{|\ol{\pi}|}G_{\ol{\pi}}$ terms over all choices of $\ol{\pi}$, as claimed.
\end{proof}

We give an example now to illustrate the formula in the smallest nontrivial case, and in particular to show where the infinite sums come from.

\begin{example}
    If $G=v$ is a single vertex $v$, the antipode needs to satisfy $$ \ol{S}(v) \sqcup v+\ol{S}(v)\sqcup\varnothing+\ol{S}(\varnothing)\sqcup v=0. $$ Rearranging gives $$ \ol{S}(v) = -\frac{v}{\varnothing+v}=-v+(v\sqcup v)-(v\sqcup v\sqcup v)+\cdots=\sum_{n\ge 1}(-v)^n. $$ Thus, we get an infinite alternating sum where each term is the disjoint union of some number of isolated vertices, since every element of $\ol{\pi}$ needs to be a single vertex, and thus each term is $(-1)^{|\ol{\pi}|}$ times the disjoint union of $|\ol{\pi}|$ copies of the vertex.
\end{example}


Humpert and Martin give the following cancellation-free formula for the antipode of $\WGraphs$:

\begin{theorem}[Humpert--Martin \cite{humpert2012incidence}, Theorem 3.1]\label{thm:humpert-martin}
    The antipode of $\WGraphs$ is given by $$ S(G) = \sum_{F\tn{ a flat}}(-1)^{c(G_F)}a(G/F)G_F, $$ where $c(G_F)$ is the number of connected components of $G_F$ and $a(G/F)$ is the number of acyclic orientations on $G/F$.
\end{theorem}

For a \emph{multiset} $\ol{F}$ of (possibly repeated and possibly overlapping) \emph{connected} subsets of $V(G)$, write $C_{\ol{F}}(G)$ for the clan graph formed by blowing up each vertex of $G$ into a clique of size equal to the number of times it occurs in $\ol{F}$, with every copy of a vertex adjacent to every copy of each of its neighbors in $G$, and write $C_{\ol{F}}(G)/\ol{F}$ for the graph obtained from $C_{\ol{F}}(G)$ by assigning a different copy of each vertex to each subset in $\ol{F}$ containing it and contracting all vertices assigned to each subset in $\ol{F}$ into a single vertex. Thus, the vertices of $C_{\ol{F}}(G)/\ol{F}$ are the subsets in $\ol{F}$, counted with multiplicity, and two of them are adjacent if and only if they share a vertex or some vertex of one is adjacent in $G$ to some vertex of the other. We also write $r_{\ol{F}}!$ for the product of the factorials of the multiplicities of the distinct subsets appearing in $\ol{F}$. We can then give the following $K$-analogue of Humpert and Martin's result:

\begin{theorem}\label{thm:humpert-martin_K}
    The antipode of $\mWGraphs$ is given by $$ \ol{S}(G) = \sum_{\ol{F}} (-1)^{|\ol{F}|}\frac1{r_{\ol{F}}!}a(C_{\ol{F}}(G)/\ol{F})G_{\ol{F}}, $$ where $\ol{F}$ ranges over all finite multisets of nonempty connected subsets of $V(G)$ whose union is $V(G)$, $G_{\ol{F}}$ is the disjoint union of the induced subgraphs of $G$ corresponding to the elements of $\ol{F}$, and $r_{\ol{F}}!$ denotes the product of the factorials of the multiplicities of each repeated part of $\ol{F}$.
\end{theorem}


\begin{proof}
    We can start with our formula from Theorem \ref{thm:schmitt_K}, and consider which terms cancel. The extra cancellation comes from the fact that in Theorem \ref{thm:humpert-martin_K}, the vertex sets in $\ol{F}$ are required to be connected, while in Theorem \ref{thm:schmitt_K}, the vertex sets in $\ol{\pi}$ need not be connected. Given a sequence $\ol{\pi}$, the corresponding multiset $\ol{F}$ must be obtained by taking the elements of $\ol{F}$ to be the connected components of the elements of $\ol{\pi}$, so that $G_{\ol{F}} = G_{\ol{\pi}}$. Thus, we need to show that the sum of $(-1)^{|\ol{\pi}|}$ over all sequences $\ol{\pi}$ corresponding to a given multiset $\ol{F}$ is always equal to $(-1)^{|\ol{F}|}\frac1{r_{\ol{F}}!}a(C_{\ol{F}}(G)/\ol{F})$. 
    
    We can show this by applying Theorems \ref{thm:schmitt} and \ref{thm:humpert-martin} to $S(C_{\ol{F}}(G)/\ol{F})$ and looking at the coefficient of the edgeless subgraph consisting of all the vertices of $C_{\ol{F}}(G)/\ol{F}$ with no edges between them. The vertices of $C_{\ol{F}}(G)/\ol{F}$ correspond to the parts of $\ol{F}$, and their disjoint union is thus the edgeless graph $G_{\ol{F}}/\ol{F}$ formed by contracting each part of $\ol{F}$ into a single vertex. 
    
    On the one hand, by Theorem \ref{thm:schmitt}, the coefficient of this graph $G_{\ol{F}}/\ol{F}$ in $S(C_{\ol{F}}(G)/\ol{F})$ is equal to the sum of $(-1)^{|\pi|}$ over all ordered vertex partitions $\pi$ of $V(C_{\ol{F}}(G)/\ol{F})$ such that $(C_{\ol{F}}(G)/\ol{F})_\pi$ is edgeless. This is $r_{\ol{F}}!$ times the sum of $(-1)^{|\ol{\pi}|}$ over all sequences of vertex subsets $\ol{\pi}$ such that the connected components of the parts of $\ol{\pi}$ are the parts of $\ol{F}$. This is because we can get from $\ol{\pi}$ to $\pi$ by contracting each part of $\ol{F}$ into a single vertex, keeping $|\pi|=|\ol{\pi}|$, except that actually each $\ol{\pi}$ corresponds to $r_{\ol{F}}!$ different choices of $\pi$, because in choosing an ordered vertex partition $\pi$, we treat the repeated parts in $\ol{F}$ as distinct vertices of $C_{\ol{F}}(G)/\ol{F}$, while in choosing an ordered sequence $\ol{\pi}$, repeated copies of the same connected component of $G$ are not distinguishable from each other, and for a repeated part of $\ol{F}$ with multiplicity $r_i$, there are $r_i!$ ways to assign the corresponding repeated vertices of $C_{\ol{F}}(G)/\ol{F}$ to the $r_i$ parts of $\pi$ corresponding to parts of $\ol{\pi}$ containing that connected component of $\ol{F}$ (since no part of $\ol{\pi}$ can contain the same connected component more than once).

    On the other hand, by Theorem \ref{thm:humpert-martin}, the $G_{\ol{F}}/\ol{F}$ term in $S(C_{\ol{F}}(G)/\ol{F})$ corresponds to taking the flat $F$ to have no edges. Then the number of connected components is just the number of vertices, which is equal to $|\ol{F}|$ since each vertex comes from contracting one of the parts of $\ol{F}$, so the sign is $(-1)^{|\ol{F}|}$. Then since $F$ contains no edges, contracting all the edges in $F$ does nothing, and the coefficient is thus $(-1)^{|\ol{F}|}a(C_{\ol{F}}(G)/\ol{F})$, which is $r_{\ol{F}}!$ times our desired coefficient of $G_{\ol{F}}$. The result follows from matching up those coefficients.
\end{proof}

The extra cancellation in Theorem \ref{thm:humpert-martin_K} relative to Theorem \ref{thm:schmitt_K} comes from the requirement that the vertex sets be connected, so it is worth looking at the two extreme cases where $G$ has all possible edges and where it has none. In the first case, no cancellation happens at all.

\begin{example}
    If $G$ is a clique, then it has no disconnected induced subgraphs, so Theorems \ref{thm:schmitt_K} and \ref{thm:humpert-martin_K} give the same sum with no extra cancellation. Each multiset of connected induced subgraphs $\ol{F}$ corresponds exactly to the $|\ol{F}|!/r_{\ol{F}}!$ sequences $\ol{\pi}$ of arbitrary induced subgraphs obtained by ordering its parts, with $|\ol{F}|=|\ol{\pi}|$, and $a(C_{\ol{F}}(G)/\ol{F})=|\ol{F}|!$ is just the number of total orderings on the parts of $\ol{F}$, since $C_{\ol{F}}(G)/\ol{F}$ is a contraction of a clan graph of a complete graph and is thus itself a complete graph, and an acyclic orientation on a complete graph is the same as a total order on its vertices.
\end{example}

At the other extreme, when $G$ has no edges at all, the antipode factors completely and we can check the formula against a direct computation.

\begin{example}
    If $G$ is edgeless, then $G$ is the product of its vertices, since our multiplication is given by taking disjoint unions. Thus, $\ol{S}(G)$ is just the product of $\ol{S}(v)$ over all vertices $v\in V(G)$: $$ \ol{S}(G) = \bigsqcup_{v\in V(G)}\ol{S}(v) = \bigsqcup_{v\in V}\left(-\frac{v}{\varnothing+v}\right) = \bigsqcup_{v\in V}(-v + (v\sqcup v)-(v\sqcup v\sqcup v)+\cdots). $$ Each term of $\ol{S}(v)$ thus corresponds to taking each vertex some number of times, with the sign equal to $-1$ to the power of the total number of vertices. This matches Theorem \ref{thm:humpert-martin_K} because for an edgeless graph, a multiset $\ol{F}$ of connected induced subgraphs is the same as a multiset of singleton vertices. Then the number of acyclic orientations on $C_{\ol{F}}(G)/\ol{F} = C_{\ol{F}}(G)$ is equal to the number of ways to order each set of repeated vertices among themselves, which is $r_{\ol{F}}!$. Thus, if we divide the number of acyclic orientations by $r_{\ol{F}}!$ we get 1, matching the fact that each multiset $\ol{F}$ of vertices of $G$ shows up exactly once in our expanded product, with sign equal to $(-1)^{|\ol{F}|}$.
\end{example}


\subsection{Heap interpretation of \texorpdfstring{$\omega(\ol{X}_G)$}{XG}}\label{subsec:heaps}

Recall that the involution $\omega$ on the ring of symmetric functions $\Sym$ (or on its completion $\mSym$) is the ring isomorphism given by $$ \omega(p_n) = (-1)^{n-1}p_n. $$ Comparing this to the antipode $$ S(p_n) = -p_n $$ on the Hopf algebra $\Sym$, we see that $\omega(p_n)=(-1)^n S(p_n)$, so that applying $\omega$ to $p_n$ is equivalent to applying $S$ and then negating the resulting degree-$n$ term. Since a monomial of degree $n$ picks up a sign $(-1)^n$ under the substitution $x_i\mapsto -x_i$, it follows that for any symmetric function $f$, $$ \omega(f)(x_1,x_2,\dots) = S(f)(-x_1,-x_2,\dots). $$ 

For unweighted graphs $G$, in \cite{pierson2025lyndon}, the second author used the heap interpretation of $\omega(\ol{X}_G)$ as the generating series for all ways to assign a (possibly empty) heap on $G$ to each color such that all vertices are used by at least one of the heaps, where a \newword{heap} on $G$ is an acyclic orientation on a clan graph of some induced subgraph of $G$, treating repeated copies of a vertex as indistinguishable from each other. Each such collection of color heaps corresponds to a monomial in $\omega(\ol{X}_G)$, and the exponent on each color variable is the total number of vertices in its heap, counted with multiplicity (called the \newword{size} of the heap). These graphical heaps are equivalent to the heaps of pieces from Viennot \cite{viennot2006heaps} and Cartier and Foata \cite{cartier2006problemes} and were also used by Bernardi and Nadeau in \cite{bernardi2020combinatorial} in a related CSF context. Using Marberg's character $\ol{\zeta}_{\Sym}$ sending the first variable $x_1$ to $t$ and all other variables to 0, we thus get that $\ol{\zeta}_{\Sym}(\omega(\ol{X}_G))$ is the generating series $H_G(t)$ for heaps on $G$ that use all vertices, where the coefficient of $t^n$ is the number of such heaps of size $n$, and thus $$ \ol{\zeta}_{\Sym}(\omega(\ol{X}_G)(-x_1,-x_2,\dots))=H_G(-t). $$ 

We can use Theorem \ref{thm:humpert-martin_K} to recover this heap interpretation $\ol{\zeta}_\Sym(\omega(\ol{X}_G))=H_G(t)$, and then to deduce the heap interpretation of $\omega(\ol{X}_G)$ as a consequence, showing some more general facts along the way.

\begin{prop}[cf. \cite{bernardi2020combinatorial,pierson2025lyndon}]\label{prop:omega(X_G)}
    For any unweighted graph $G$, $\omega(\ol{X}_G)$ is the sum of the monomials corresponding to all ways to cover $V(G)$ with a (possibly empty) heap for each of finitely many colors, such that every vertex is in at least one color heap and the exponent on each variable $x_i$ is the size of the $i^{\text{th}}$ color heap. Equivalently, $\omega(\ol{X}_G)$ is the generating series for ways to cover $G$ with a finite set of distinct nonempty heaps and then assign a nonempty finite set of colors to each heap, with different heaps getting disjoint sets of colors.
\end{prop}

We can prove Proposition \ref{prop:omega(X_G)} by first interpreting the associated character:

\begin{lemma}\label{lem:zeta(omega(XG))}
    $\ol{\zeta}_\Sym(\omega(\ol{X}_G))=H_G(t).$
\end{lemma}

\begin{proof}
    Since the $G\mapsto \ol{X}_G$ map is an LC-Hopf algebra map, it preserves both characters and antipodes, so $$ \ol{\zeta}_\Sym(\omega(\ol{X}_G)(-x_1,-x_2,\dots))=\ol{\zeta}_{\Sym}(S(\ol{X}_G))=(\ol{\zeta}_\Sym)^{-1}(\ol{X}_G)=(\ol{\zeta}_{\WGraphs})^{-1}(G)=\ol{\zeta}_{\WGraphs}(\ol{S}(G)), $$ where by the inverse we mean the inverse element in the group of characters, with product given by convolution. By Theorem \ref{thm:humpert-martin_K} and linearity of characters, $$ \ol{\zeta}_{\WGraphs}(\ol{S}(G)) = \sum_{\ol{F}}(-1)^{|\ol{F}|}\frac1{r_{\ol{F}}!}a(C_{\ol{F}}(G)/\ol{F})\ol{\zeta}_\WGraphs(G_{\ol{F}}). $$ Now, $\ol{\zeta}_\WGraphs(G_{\ol{F}})=0$ unless $G_{\ol{F}}$ is edgeless, so since the parts of $\ol{F}$ are the connected components of $G_{\ol{F}}$, this means all parts of $\ol{F}$ must be singleton vertices. In that case, $\ol{\zeta}_\WGraphs(G_{\ol{F}})=t^{|V(G_{\ol{F}})|}=t^{|\ol{F}|}.$ Then contracting each part of $\ol{F}$ into a single vertex does not do anything, since the parts of $\ol{F}$ are already singleton vertices, and choosing $\ol{F}$ just corresponds to choosing some composition $\alpha$ with all positive parts that represent the multiplicity of each vertex of $G$ in $\ol{F}$, and we write $C_\alpha(G)$ for the clan graph $C_{\ol{F}}(G)=C_{\ol{F}}(G)/\ol{F}$. Putting this together, we get $$ \ol{\zeta}_{\WGraphs}(\ol{S}(G)) = \sum_{\alpha}(-t)^{|\alpha|}\frac1{\alpha!}a(C_\alpha(G)). $$ But $\frac1{\alpha!}a(C_\alpha(G))$ is equal to the number of heaps on the $\alpha$-clan graph $C_\alpha(G)$ (called \newword{heaps of type $\alpha$}), because a heap is just an acyclic orientation on some clan graph, and the $\frac1{\alpha!}$ accounts for the fact that in a heap, we want to treat repeated copies of the same vertex as indistinguishable from each other, and hence need to divide the number of acyclic orientations by $\alpha_v!$ for each $v$, since $\alpha_v$ is the number of copies of the vertex $v$ in $C_\alpha(G)$. Since $|\alpha|=|V(C_\alpha(G))|$ is the size of each heap on $C_\alpha(G)$, the coefficient of $(-t)^n$ is equal to the total number of heaps of size $n$ on $G$. Thus, $\ol{\zeta}_\Sym(\omega(\ol{X}_G)(-x_1,-x_2,\dots))=H_G(-t),$ so negating all the variables, $\ol{\zeta}_\Sym(\omega(\ol{X}_G))=H_G(t).$
\end{proof}

The heap interpretation of $\omega(\ol{X}_G)$ in Proposition \ref{prop:omega(X_G)} is then a fairly immediate consequence of Lemma \ref{lem:zeta(omega(XG))}, because of the following more general fact that lets us turn an interpretation of a character on $\mWGraphs$ into an interpretation of the corresponding map to $\mSym$:

\begin{prop}\label{prop:mWGraphs_map_interpretation}
    Suppose we have a character $\zeta:\mWGraphs\to\mb{K} \ldb t \rdb $ that is the generating series for some sort of (possibly signed) combinatorial structures on the graph $G$, where each monomial has sign corresponding to the sign of its associated structure, and exponent corresponding to some measure of the ``size" of the structure. Then the corresponding image of $G$ in $\mSym$ is the generating series for ways to choose an induced subgraph of $G$ for each color, with all but finitely many of them empty and with every vertex of $G$ in at least one of them, and choose a structure on each, such that the exponent on each color variable $x_i$ is the size of the associated structure, and the sign of each monomial is the product of the signs of its associated structures.
\end{prop}

\begin{proof}
    From Marberg \cite{marberg2023kromatic}, the image in $\mSym$ of $G\in\mWGraphs$ can be computed by applying $\blacktriangle^{(k)}$ to $G$ (which means repeatedly applying $\blacktriangle$ to the first factor in each current tensor term until all the tensors have $k$ total factors), then sending each tensor to the product $$ G_1\otimes \dots \otimes G_k\mapsto (\zeta(G_1)|_{t=x_1})\cdots (\zeta(G_k)|_{t=x_k}), $$ and then taking the limit as $k\to\infty$ of the sum of all such monomials. The $k$-fold tensor product for $\blacktriangle$ is $$ \blacktriangle^{(k)}G = \sum_{S_1\cup\dots\cup S_k=V(G)}G|_{S_1}\otimes \dots \otimes G|_{S_k}. $$ Thus, the monomials in the image of $G$ in $\mSym$ are precisely the monomials coming from one of the products $$ (\zeta(G|_{S_1})|_{t=x_1})\dots(\zeta(G|_{S_k})|_{t=x_k}), $$ corresponding to covering $G$ with the induced subgraphs $G|_{S_1},\dots,G|_{S_k}$, and then choosing a structure colored with color $i$ on the $i^{\text{th}}$ induced subgraph $G|_{S_i}$.
\end{proof}

The $G\mapsto\ol{X}_G$ map is the case corresponding to the character $\ol{\zeta}_{\WGraphs}$ where the structures are independent sets, with size given by the sum of the weights of the vertices in the independent set. For unweighted $G$, the $G\mapsto S(\ol{X}_G)$ map is the case corresponding to the character $(\ol{\zeta}_{\WGraphs})^{-1}(G)=\ol{\zeta}_{\WGraphs}(\ol{S}(G))=H_G(-t),$ so $S(\ol{X}_G)$ is the generating series for ways to cover $G$ with colored signed heaps, with size given by the sum of the vertex weights within the heap, and sign equal to $(-1)^{\tn{heap size}}.$ Then $\omega(\ol{X}_G)$ corresponds to plugging in the negatives of all the $x_i$ variables in $S(\ol{X}_G)$, so it is the generating series for ways to cover $G$ with colored unsigned heaps, as claimed in Proposition \ref{prop:omega(X_G)}.

\subsection{Factorizations using grouplike elements}\label{subsec:grouplike}

Next, we give a Hopf algebra interpretation for some of the second author's calculations and results from \cite{pierson2025lyndon} about power sum expansions for $\ol{X}_G$ and $\omega(\ol{X}_G)$. An element $f$ of a Hopf algebra is \newword{primitive} if $\Delta f = f\otimes 1 + 1\otimes f$, and \newword{grouplike} if $f\ne 0$ and $\Delta f = f\otimes f$. A \newword{Schauder basis} for an LC-Hopf algebra is a set of elements such that every element is a unique but possibly infinite linear combination of the elements in the Schauder basis. Our main idea now will essentially be to break our elements down into (possibly infinite) linear combinations of grouplike elements, because the images in $\mSym$ of those elements can be conveniently factored. The following proposition collects everything we will need about grouplike elements: part (a) is how we will recognize them, part (b) is the factorization we are after, part (c) is what makes that factorization enough to pin down an entire map to $\mSym$, and part (d) is the special case $A=\mSym$, which is the form we will use when we compare factorizations of symmetric functions against each other.

\begin{prop}\label{prop:grouplike}
    Let $A$ be a commutative LC-Hopf algebra with character $\ol{\zeta}_A:A\to\mb{K} \ldb t \rdb $, and for $f\in A$ write  $\ln(f):=-\sum_{n\ge 1}\frac1n(1-f)^{n}$ and $\exp(f):=\sum_{n\ge 0}\frac1{n!}f^n $ whenever these series are defined.
    \begin{enumerate}[label=\textup{(\alph*)}]
        \item An element $f\in A$ for which $\ln(f)$ is defined is grouplike if and only if $\ln(f)$ is primitive. Equivalently, the grouplike elements whose logarithms are defined are exactly the exponentials of the primitive elements.
        \item Suppose $A$ is cocommutative, and let $\phi$ be the LC-Hopf algebra map from $(A,\ol{\zeta}_A)$ to $(\mSym,\ol{\zeta}_\Sym)$, which is guaranteed from \cite{marberg2023kromatic} to be unique. Then for any grouplike $f\in A$ whose logarithm is defined, the image of $f$ in $\mSym$ factors over variables as $$ \phi(f)(x_1,x_2,\dots) = \prod_{i\ge 1}\ol{\zeta}_A(f)(x_i). $$ That is, we can find the symmetric function $\phi(f)$ by plugging in each $x_i$ variable in place of $t$ in the single-variable power series $\ol{\zeta}_A(f)(t)$, and then multiplying over the $x_i$ variables.
        \item If $A$ has a Schauder basis $f_1,f_2,\dots$ of grouplike elements whose logarithms are defined, then $A$ is cocommutative, so that (b) applies to each $f_n$, and the resulting values $\phi(f_n)=\prod_{i\ge 1}\ol{\zeta}_A(f_n)|_{t=x_i}$ determine $\phi$ by linearity and continuity.
        \item Taking $A=\mSym$, so that $\phi$ is the identity map, the grouplike elements of $\mSym$ are precisely the nonzero elements $g$ that factor over variables as $g=\prod_{i\ge 1}\ol{\zeta}_{\Sym}(g)|_{t=x_i}$, and they are also precisely the exponentials of the primitive elements (which are the infinite linear combinations of power sums $p_n$). They also all have constant term 1.
    \end{enumerate}
\end{prop}

\begin{proof}
    For (a), let $g:=\ln(f)$, so that $f =\exp(g)$. Assume $g$ is primitive, so $\Delta g = g\otimes1 + 1\otimes g$. Using the fact that the coproduct map is an algebra morphism, we get
    \begin{align*}
        \Delta f = \Delta \exp(g) &=\Delta\left(\sum_{n\ge 0}\frac1{n!}g^n\right)=\sum_{n\ge 0}\frac1{n!}(\Delta g)^n=\sum_{n\ge 0}\frac1{n!}(g\otimes 1+1\otimes g)^n \\
        &= \sum_{n\ge 0}\frac1{n!}\sum_{k=0}^n \binom nk g^k\otimes g^{n-k}=\sum_{n\ge 0}\sum_{k=0}^n\left(\frac1{k!}g^k\right)\otimes\left(\frac1{(n-k)!}g^{n-k}\right) \\
        &= \left(\sum_{k\ge 0}\frac1{k!}g^k\right)\otimes \left(\sum_{\ell \ge 0}\frac1{\ell!}g^\ell\right)=\exp(g)\otimes \exp(g)=f\otimes f,
    \end{align*}
    and $f\ne 0$ since $\epsilon(f)=\exp(\epsilon(g))=1$, as $\epsilon(g)=0$ by the counit axiom.
    On the other hand, assuming $\Delta f=f\otimes f$, we get $\Delta(1-f)=\Delta 1- \Delta f = 1\otimes 1-f\otimes f.$ Then we can expand
    \begin{align*}
        \Delta g = \Delta \ln(f) = \Delta\left(-\sum_{n\ge 1}\frac1n(1-f)^{n}\right)&=-\sum_{n\ge 1}\frac1n(\Delta(1-f))^{n}=-\sum_{n\ge 1}\frac1n(1\otimes 1-f\otimes f)^{n}\\
        &= -\sum_{n\ge 1}\frac1n\bigl(1\otimes 1-(f\otimes 1)(1\otimes f)\bigr)^{n}=\ln\bigl((f\otimes 1)(1\otimes f)\bigr)\\&=\ln(f\otimes 1)+\ln(1\otimes f)=\ln(f)\otimes 1+1\otimes \ln(f)\\&=g\otimes 1+1\otimes g,
    \end{align*}
    hence $g=\ln(f)$ is primitive, as needed.

    For (b), since $f$ is grouplike, $\ln(f)$ is primitive by (a), so $\Delta(\ln(f))=\ln(f)\otimes 1 + 1\otimes \ln(f),$ and applying the Hopf algebra map $\phi$ gives $$ \Delta(\phi(\ln(f))) = \phi(\ln(f))\otimes1+1\otimes\phi(\ln(f)), $$ so $\phi(\ln(f))=\ln(\phi(f))$ is also primitive. The primitive elements in $\mSym$ are precisely the (possibly infinite) linear combinations of the power sums $p_n$, so we can write $\ln(\phi(f))=\sum_{n\ge 1}a_n p_n$. Since $\ol{\zeta}_\Sym(p_n) = t^n$, the character evaluated at $\ln(\phi(f))$ is then $$ \ln(\ol{\zeta}_A(f))=\ol{\zeta}_A(\ln(f))=\ol{\zeta}_\Sym(\ln(\phi(f)))=\ol{\zeta}_\Sym\left(\sum_{n\ge 1}a_np_n\right) = \sum_{n\ge 1}a_nt^n. $$ Thus, we can write the primitive element $\ln(\phi(f))$ in $\mSym$ as the sum of its character evaluated at each of the $x_i$ variables: $$ \ln(\phi(f))=\sum_{n\ge 1}a_np_n =\sum_{i\ge 1}\sum_{n\ge 1}a_nx_i^n =\sum_{i\ge 1}\ln(\ol{\zeta}_A(f))(x_i). $$ Taking exponentials turns the sum into a product, giving the claimed formula.

    For (c), we know that every element of $A$ can be written as a possibly infinite linear combination $f=\sum_{n\ge 1}a_nf_n$. Then by linearity of the coproduct, $\Delta f = \sum_{n\ge 1}a_n\Delta f_n =\sum_{n\ge 1}a_nf_n\otimes f_n,$ which stays the same when the order of the factors is swapped within each tensor product. Thus, $A$ is cocommutative, so there is an induced map $(A,\ol{\zeta}_A)\to (\mSym,\ol{\zeta}_\Sym)$, and such a map is uniquely defined by where it sends the $f_n$'s, which is given by (b).

    For (d), take $A=\mSym$. First, applying the counit axiom $(\epsilon\otimes\tn{id})\Delta g=g$ to a grouplike $g$ gives $\epsilon(g)g=g$, so $\epsilon(g)=1$ since $g\ne 0$, and since the counit of $\mSym$ picks out the constant term, every grouplike element of $\mSym$ has constant term 1, so its logarithm is defined. There is a unique LC-Hopf algebra map from $(\mSym,\ol{\zeta}_{\Sym})$ to itself, which must thus be the identity map, so if $g\in\mSym$ is grouplike then (b) says that $g=\phi(g)$ factors over variables as $g=\prod_{i\ge 1}\ol{\zeta}_{\Sym}(g)|_{t=x_i}$, and (a) says that $g$ is the exponential of a primitive element. Conversely, suppose $g\ne 0$ factors over variables in this way. Then the power series $\ol{\zeta}_{\Sym}(g)$ must have constant term 1, or else the product defining $g$ would be 0 or not well-defined, so we can write its logarithm as another formal power series with constant term 0, $$ \ln(\ol{\zeta}_{\Sym}(g))=\sum_{n\ge 1}a_n t^n, $$ where $a_1,a_2,\dots\in\mb{K}$ are scalars. Taking logarithms in the factorization then gives $$ \ln(g)=\sum_{i\ge 1}\ln(\ol{\zeta}_{\Sym}(g))|_{t=x_i}=\sum_{i\ge 1}\sum_{n\ge 1}a_nx_i^n=\sum_{n\ge 1}a_n p_n. $$ Since the power sums are by definition primitive elements in $\mSym$, any linear combination of them is also primitive, so $\ln(g)$ is well-defined and primitive, and hence $g$ is grouplike by (a).
\end{proof}

Again assuming we are working in a commutative LC-Hopf algebra, we also have $S(f) = -f$ for primitive $f$, which by Proposition \ref{prop:grouplike}(a) implies $S(\exp(f)) = \exp(S(f))=\exp(-f)=\frac1{\exp(f)}$ for grouplike $\exp(f)$, by virtue of the fact that $S$ is an algebra homomorphism (since our Hopf algebra is commutative), and hence $S$ commutes with any polynomial or power series map from the Hopf algebra to itself, and hence in particular with the $\exp$ map. In other words, applying the antipode to a grouplike element takes its reciprocal.

For the case $A=\mWGraphs$, there is a specific set of grouplike generators that will be useful to us:

\begin{prop}\label{prop:mWGraphs_primitives}
    Write $$ Y(G):= \sum_{W\se V(G)}G|_W $$ for the sum of all the induced subgraphs of $G$. Then $Y(G)$ is grouplike in $\mWGraphs$, and the $Y(G)$'s form a basis for the finite linear combinations of graphs.
\end{prop}

\begin{proof}
    To show that $\blacktriangle Y(G)=Y(G)\otimes Y(G),$ from the coproduct formula for $\mWGraphs$, we get
    \begin{align*}
        \blacktriangle Y(G)&=\sum_{W\se V(G)}\sum_{S\cup T=W}(G|_W)|_S\otimes (G|_W)|_T = \sum_{S\cup T\se V(G)}G|_S\otimes G|_T \\
        &= \left(\sum_{S\se V(G)}G|_S\right)\otimes \left(\sum_{T\se V(G)}G|_T\right)
        = Y(G)\otimes Y(G),
    \end{align*}
    as needed. To see that the $Y(G)$'s form such a basis, note that $Y(G)$ is $G$ plus a linear combination of graphs with fewer vertices, so it suffices to show that every graph is a finite linear combination of $Y(G)$'s. This works because we can write $$ G=\sum_{W\se V(G)}(-1)^{|V(G)\bs W|}Y(G|_W) $$ due to the fact that we can expand the right side as $$ \sum_{W\se V(G)}(-1)^{|V(G)\backslash W|}\sum_{W'\se W}G|_{W'}. $$ Then we get exactly one term equal to the graph $G$ itself, by setting $W'=W=V(G)$. For each $W'\subsetneq V(G)$, we get a $G|_{W'}$ term for each subset $W$ with $W'\se W$ and $W\se V(G)$. Fixing some particular vertex $v\not\in W'$, we get a sign-reversing involution by pairing each subset $W$ containing $v$ with the subset $W\backslash\{v\}$ with $v$ removed, and pairing each subset $W$ not containing $v$ with the subset $W\cup \{v\}$ with $v$ added. Thus, all the $G|_{W'}$ terms cancel when $W'\ne V(G)$, so the only remaining term is $G|_{V(G)}=G$, as claimed. Thus, the $Y(G)$'s form a basis of grouplike elements for the finite linear combinations of graphs.
\end{proof}

Feeding the $Y(G)$'s into Proposition \ref{prop:grouplike} gives us an explicit way to write the map $\mWGraphs\to\mSym$ corresponding to a given character:

\begin{cor}\label{cor:mWGraphs_map_given_character}
    The map $\mWGraphs\to\mSym$ corresponding to a given character $\zeta:\mWGraphs \to \mb{K} \ldb t \rdb $ is given by $$ G\mapsto \sum_{W\se V(G)}(-1)^{|V(G)\backslash W|}\prod_{i\ge 1}\zeta(Y(G|_W))|_{t=x_i}. $$
\end{cor}

\begin{proof}
    By Proposition \ref{prop:mWGraphs_primitives}, each $Y(G)$ is grouplike, and its logarithm is defined, since every term of $Y(G)$ other than $\varnothing$ has at least one vertex. As $\mWGraphs$ is cocommutative, Proposition \ref{prop:grouplike}(b) then shows that the map induced by $\zeta$ sends $Y(G)\mapsto\prod_{i\ge 1}\zeta(Y(G))|_{t=x_i}$. Applying this to the expansion $G=\sum_{W\se V(G)}(-1)^{|V(G)\bs W|}Y(G|_W)$ from the proof of Proposition \ref{prop:mWGraphs_primitives} and using linearity gives the stated formula.
\end{proof}

From the definition of $Y(G)$ and linearity of characters, we can write $$ \zeta(Y(G))=\sum_{W\se V(G)}\zeta(G|_W). $$ In the case of $\ol{\zeta}_{\WGraphs}$, we get that $\ol{\zeta}_{\WGraphs}(Y(G))$ is the \newword{independence polynomial} $I_G(t)$, whose $t^n$ coefficient is the number of independent sets in $G$ of weight $n$, hence we recover the formula $$ \ol{X}_G = \sum_{W\se V(G)}(-1)^{|V(G)\bs W|}\prod_{i\ge 1}I_{G|_W}(x_i) $$ from \cite{pierson2025lyndon}. Similarly, for $G$ unweighted, if we use the inverse character, which sends $Y(G)\mapsto 1/I_G(t) = \sum_{W\se V(G)}H_{G|_W}(-t)$, we get $$ S(\ol{X}_G)=\sum_{W\se V(G)}(-1)^{|V(G)\bs W|}\prod_{i\ge 1}\frac1{I_{G|_W}(x_i)}, $$ and we can then replace $S$ with $\omega$ by negating all the variables to get $$ \omega(\ol{X}_G)=\sum_{W\se V(G)}(-1)^{|V(G)\bs W|}\prod_{i\ge 1}\frac1{I_{G|_W}(-x_i)}, $$ which was the other key formula used in \cite{pierson2025lyndon}, since $1/I_{G|_W}(-t)$ is the generating series for all heaps on $G|_W$.

\subsection{Hopf algebra interpretation of the power sum expansion}\label{subsec:p}

The purpose of introducing these formulas in \cite{pierson2025lyndon} was to give expansion formulas for $\ol{X}_G$ and $\omega(\ol{X}_G)$ in terms of the $K$-theoretic power sum bases $\{\ol{p}_\lam\}$ and $\{\ol{p}'_\lam\}$ from Section~\ref{sec:power-sum}, and those methods can also be generalized and given a Hopf algebra interpretation. Essentially, what made these power sum bases nice is that the elements $1+\ol{p}_n$ and $1+\ol{p}'_n$ are also grouplike, which by Proposition \ref{prop:grouplike}(d) means we can lift a factorization at the character level, such as $$ I_{G|_W}=\prod_{n\ge 1}(1+t^n)^{a_W(n)} $$ to a corresponding factorization of the grouplike terms at the symmetric function level, $$ \ol{X}_G=\sum_{W\se V(G)}(-1)^{|V(G)\bs W|}\prod_{n\ge 1}(1+\ol{p}_n)^{a_W(n)}. $$

Since a grouplike element is determined by the single-variable power series given by its character, we can match that power series one degree at a time against any sequence of grouplike elements with one new degree appearing in each factor. This is what lets us turn a factorization of a character into a factorization of a symmetric function:

\begin{prop}\label{prop:mSym_grouplike_factor}
     If we have any sequence of grouplike elements $1+g_1,1+g_2,\dots$ in $\mSym$ such that the lowest degree term in $g_n$ has degree $n$, then every grouplike element $g\in \mSym$ can be uniquely written in the form $\prod_{n\ge 1}(1+g_n)^{a_n}$ for some sequence of exponents $a_1,a_2,\dots\in\mathbb{K}$, where $(1+g_n)^{a_n}=\sum_{k\ge 0}\binom{a_n}k g_n^k$.
\end{prop}

\begin{proof}
    From Proposition \ref{prop:grouplike}(d), we can factor $g$ as $g=\prod_{i\ge 1}\ol{\zeta}_{\Sym}(g)|_{t=x_i}$, where $\ol{\zeta}_{\Sym}(g)$ is a power series in $t$ with constant term 1. Then we can recursively find the exponents $a_1,a_2,\dots$ in order by choosing the exponent $a_n$ so that the $t^n$ coefficient in the product $\prod_{n\ge 1}(1+\ol{\zeta}_{\Sym}(g_n))^{a_n}$ matches the $t^n$ coefficient of $\ol{\zeta}_{\Sym}(g)$, and we have a unique choice for $a_n$ at each such step. Thus, the exponents $a_n$ are uniquely determined. We can then lift this to a factorization of $g$ by multiplying over the $x_i$ variables.
\end{proof}

The results in \cite{pierson2025lyndon} hinged heavily on the fact that for $g_n=\ol{p}_n$ or $g_n=\ol{p}'_n$ and $g=\phi(Y(G))$ or $g=\omega(\phi(Y(G)))$ with $G$ unweighted, where $\phi:\mWGraphs\to\mSym$ is the map induced by $\ol{\zeta}_{\WGraphs}$, the resulting exponents have combinatorial interpretations in terms of \newword{Lyndon heaps}. We will not define Lyndon heaps here, but we note that they play the role of the primitive structures described below for heaps, since every heap can be uniquely factored into a multiset of Lyndon heaps.

More generally, if we have a character $\zeta$ on $\mWGraphs$ such that $\zeta(G)$ is the generating series for some sort of structures on $G$, then $\zeta(Y(G))$ is the corresponding generating series for those structures on either $G$ or an induced subgraph of $G$. If there is a set of primitive structures (not to be confused with primitive elements of a Hopf algebra), each using some set of vertices of $G$, such that for every $W\se V(G)$, the possible structures on $G|_W$ or an induced subgraph of it are in a bijection with the finite sets of distinct primitive structures using only vertices in $W$ (with the sizes of the structures summing over unions), then we can write $$ \zeta(Y(G))=\prod_{n\ge 1}(1+t^n)^{a_G(n)}, $$ where $a_G(n)$ is the number of possible primitive structures of size $n$ on $G$ or on an induced subgraph of $G$. In particular, each $(1+t^n)$ factor can be thought of as corresponding to one of the primitive structures of size $n$, so choosing the 1 from that factor corresponds to not including that particular primitive structure in the set, while choosing the $t^n$ corresponds to including it. Then if $\phi:\mWGraphs\to \mSym$ is the map induced by the character $\zeta$, we can use Proposition \ref{prop:mSym_grouplike_factor} to lift this character factorization to a factorization $$ \phi(Y(G))=\prod_{n\ge 1}(1+\ol{p}_n)^{a_G(n)}, $$ which then gives $$ \phi(G)=\sum_{W\se V(G)}(-1)^{|V(G)\bs W|}\prod_{n\ge 1}(1+\ol{p}_n)^{a_{G|_W}(n)}. $$ As we saw earlier, $\phi(G)$ itself can be interpreted as the generating series for all ways to cover $G$ with some multiset of the structures on induced subgraphs of $G$, such that all vertices of $G$ get used, and such that each structure is assigned a color, where the exponent on the color variable is the size of the structure. Expanding, we see that the coefficient of $\ol{p}_\lam$ in $\phi(G)$ enumerates ways to cover $G$ with a set of distinct primitive structures of sizes $\lam_1,\dots,\lam_{\ell(\lam)},$ such that all vertices of $G$ are used by at least one of the structures (because the alternating sum over induced subgraphs cancels out cases where all of $\lam_1,\dots,\lam_{\ell(\lam)}$ are structures on the same proper induced subgraph). Since applying the antipode to a grouplike element takes its reciprocal, as we saw following Proposition \ref{prop:grouplike}, we automatically get $$ S(\phi(G)) =\phi(\ol{S}(G)) = \sum_{W\se V(G)}(-1)^{|V(G)\bs W|}\prod_{n\ge 1}(1+\ol{p}_n)^{-a_{G|_W}(n)}. $$ Then we automatically get a corresponding interpretation for the $\ol{p}_\lam$ coefficients of $S(\phi(G))$: $(-1)^{\ell(\lam)}[\ol{p}_\lam]S(\phi(G))$ is equal to the number of ways to choose a multiset of the primitive structures with sizes equal to the parts of $\lam$ such that every vertex is used by at least one of the structures, using the fact that for $k\ge 0$, $$ (1+x)^{-k} = \sum_{i\ge 0}(-1)^i x^i\left(\!\!\binom{k}{i}\!\!\right), $$ where $\left(\!\binom ki \!\right)$ is the number of ways to choose a multiset of $i$ elements out of $k$ options.

Similarly, if our structures instead split into primitive structures such that the possible structures on induced subgraphs of $G$ are in a bijection with the possible finite multiset of primitive structures, then we can instead write $$ \zeta(Y(G))=\prod_{n\ge 1}(1+t^n+t^{2n}+t^{3n}+\dots)^{a_G(n)}, $$ where again $a_G(n)$ counts the number of possible primitive structures of size $n$ on induced subgraphs of $G$. This follows because again each factor corresponds to one particular primitive structure, and choosing the $t^{kn}$ term from such a factor represents including $k$ copies of that particular size $n$ primitive structure in the multiset. This then similarly lifts to a symmetric function factorization: $$ \phi(G)=\sum_{W\se V(G)}(-1)^{|V(G)\bs W|}\prod_{n\ge 1}(1+\ol{p}_n')^{a_{G|_W}(n)}. $$ We can then interpret the $\ol{p}'_\lam$ coefficient in exactly the same way, as the number of ways to cover $G$ with distinct primitive structures of size $\lam_1,\dots,\lam_{\ell(\lam)}$ on subsets of the vertices such that each vertex gets used by at least one of the primitive structures, and we can similarly interpret the $\ol{p}'_\lam$ coefficient in $S(\phi(G))$ as being $(-1)^{\ell(\lam)}$ times the number of ways to cover $G$ with a multiset of the primitive structures with sizes equal to the parts of $\lam$ such that each vertex is used at least once.

Whether the $\ol{p}_\lam$ or $\ol{p}'_\lam$ basis is more natural thus essentially depends on whether the structures split more naturally into sets of distinct primitive structures, or into multisets of primitive structures. In the case of heaps, they split more naturally into multisets of primitive structures (the primitive structures being Lyndon heaps), so the $\ol{p}'_\lam$ basis is a more natural choice than the $\ol{p}_\lam$ basis.

\subsection{Hopf algebra interpretation of the monomial expansion}\label{subsec:m}

The same grouplike factorization ideas also give a clean interpretation of the monomial expansion of $\ol{X}_G$. Recall from Section~\ref{sec:background} the $K$-theoretic augmented monomial symmetric functions $\m_\lam=\ol{X}_{K_\lam}$, where $K_\lam$ is the complete graph on $\ell(\lam)$ vertices whose vertex weights are the parts $\lam_1,\dots,\lam_{\ell(\lam)}$ of $\lam$, and the product $\odot$ with $\m_\lam\odot\m_\mu=\m_{\lam\sqcup\mu}$. The proper set colorings of $K_\lam$ are the ways to choose pairwise disjoint nonempty finite sets of colors for its vertices, so expanding in the augmented monomial basis shows that $\m_\lam$ is the product, under Tsujie's $\odot$ from Section~\ref{sec:background}, of the elements $\m_{\lam_i}=\sum_{k\ge 1}\augm{(\lam_i^k)}/k!$, where $(\lam_i^k)$ is the partition with $k$ parts equal to $\lam_i$. Thus, $\odot$ is just Tsujie's $\odot$, extended to infinite linear combinations. Since the coproduct of $\augm{\lam}$ is the sum of $\augm{\mu}\otimes\augm{\nu}$ over all ways to split the parts $\lam_1,\dots,\lam_{\ell(\lam)}$ between two partitions $\mu$ and $\nu$, it follows that $\Delta(f\odot g)=\Delta(f)\odot\Delta(g)$, where $\odot$ acts on each tensor factor separately, and hence the $\odot$-product of grouplike elements is grouplike.

Now let $\phi:\mWGraphs\to\mSym$ be the map induced by $\ol{\zeta}_{\WGraphs}$, so that $\phi(Y(G))=\sum_{W\se V(G)}\ol{X}_{G|_W}$, which is $Y_G$ from Definition~\ref{def:Y}. Since $Y(G)$ is grouplike, so is $\phi(Y(G))$, and its character is $\ol{\zeta}_{\WGraphs}(Y(G))=I_G(t)$. On the other hand, for each nonempty stable set $S$ of $G$, with total weight $w(S)$, the element $1+\m_{w(S)}=1+\ol{p}_{w(S)}$ is grouplike, so the $\odot$-product of these elements is grouplike too, and setting $x_1=t$ and all other variables to 0 sends it to $1+\sum_S t^{w(S)}=I_G(t)$, since $\augm{\lam}\odot\augm{\mu}=\augm{\lam\sqcup\mu}$ only has monomials in at least two variables when $\lam$ and $\mu$ are nonempty. So by Proposition \ref{prop:grouplike}(d), $$ \phi(Y(G))=\bigodot_{S\tn{ a stable set}}\left(1+\m_{w(S)}\right)=\sum_{\substack{S\text{ a set of}\\\text{distinct stable sets}}}\m_{\lam(S)}, $$ where the sum is over all finite sets $S$ of distinct nonempty stable sets of $G$, and $\lam(S)$ is the partition whose parts are the weights of the stable sets in $S$. Applying the alternating sum over induced subgraphs from Corollary~\ref{cor:mWGraphs_map_given_character} then recovers the stable set cover expansion $$ \ol{X}_G=\sum_{W\se V(G)}(-1)^{|V(G)\bs W|}\phi(Y(G|_W))=\sum_{\mathcal{C}\in\textsf{\textup{SSC}}(G)}\m_{\lam(\mathcal{C})}, $$ where the alternating sum over $W$ has the effect of restricting to those collections of distinct stable sets that actually cover all of $V(G)$, which are exactly the stable set covers $\mathcal{C}$ of $G$. This recovers the monomial expansion of the KSF from \cite{crew2023kromatic} through the grouplike factorization machinery.

\section*{Acknowledgements}

We thank José Aliste-Prieto for suggesting the idea for the Hopf algebra part of this paper. 

\section*{AI statement}

No AI was used at any point in this research.

\printbibliography

\end{document}